\documentclass[]{article}

\title{Logarithmic corrections to the Alexander--Orbach conjecture for the four-dimensional uniform spanning tree}
\author{Noah Halberstam \and Tom Hutchcroft}

\usepackage{amsfonts,amsmath,amssymb,amsthm}
\usepackage[utf8]{inputenc}
\usepackage[english]{babel}
\usepackage{bbm}
\usepackage{enumitem}
\usepackage{color}
\usepackage{comment}
\usepackage[margin=1.0in]{geometry}
\usepackage{setspace}
\usepackage{scalerel}
\usepackage{mathrsfs}
\usepackage{multicol}
\usepackage{hyperref}
\usepackage{cite}

\newtheorem{theorem}{Theorem}

\newtheorem{proposition}[theorem]{Proposition}

\newtheorem{lemma}[theorem]{Lemma}
\theoremstyle{remark}
\newtheorem{remark}{Remark}

\DeclareFontFamily{U}{mathb}{\hyphenchar\font45}
\DeclareFontShape{U}{mathb}{m}{n}{
	<5> <6> <7> <8> <9> <10> gen * mathb
	<10.95> mathb10 <12> <14.4> <17.28> <20.74> <24.88> mathb12
}{}
\DeclareSymbolFont{mathb}{U}{mathb}{m}{n}
\DeclareMathSymbol{\llcurly}{3}{mathb}{"CE}

\numberwithin{theorem}{section}

\DeclareMathOperator{\Z}{\mathbb{Z}}

\DeclareMathOperator{\N}{\mathbb{N}}

\DeclareMathOperator{\Rext}{\operatorname{rad}_{\mathrm{ext}}}

\DeclareMathOperator*{\argmax}{arg\,max}

\newcommand{\E}[1]{\mathbb{E}\left[#1\right]}

\DeclareMathOperator{\pr}{\mathbb{P}}
\DeclareMathOperator{\pb}{\mathbf{P}}

\newcommand{\floor}[1]{\lfloor#1\rfloor}
\newcommand{\ceil}[1]{\lceil#1\rceil}

\newcommand{\abs}[1]{\lvert#1\rvert}

\newcommand{\norm}[1]{\left\lVert#1\right\rVert}

\newcommand{\fB}{\mathfrak{B}}
\newcommand{\fT}{\mathfrak{T}}
\newcommand{\fP}{\mathfrak{P}}
\newcommand{\bP}{\mathbf{P}}
\newcommand{\bE}{\mathbf{E}}
\newcommand{\bO}{\mathbf{O}}
\newcommand{\bo}{\mathbf{o}}
\newcommand{\bOmega}{\mathbf{\Omega}}

\newcommand{\bTheta}{\mathbf{\Theta}}
\newcommand{\eps}{\varepsilon}
\renewcommand{\P}{\mathbb{P}}

\newcommand{\LE}{\mathrm{LE}}

\begin{document}

\maketitle

\begin{abstract}
	We compute the precise logarithmic corrections to Alexander--Orbach behaviour for various quantities describing the geometric and spectral properties of the four-dimensional uniform spanning tree. In particular, we prove 
   that the volume of an intrinsic $n$-ball in the tree is $n^2 (\log n)^{-1/3+o(1)}$, that the typical intrinsic displacement of an $n$-step random walk is $n^{1/3} (\log n)^{1/9-o(1)}$, and that the $n$-step return probability of the walk decays as $n^{-2/3}(\log n)^{1/9-o(1)}$. 
\end{abstract}



\section{Introduction}
The behaviour of random walks on random fractals has been the subject of intense
 study since the 1970s~\cite{de1976percolation}, and a sophisticated and widely applicable theory  has now developed on the topic \cite{BJKS08,KumagaiBook,kestensubdiff}.
   In particular, it is now well established that the asymptotic behaviour of 
    spectral quantities such as exit times, return probabilities, and walk displacement are determined under mild conditions by geometric properties such as volume growth and resistance growth
       \cite{MR1881970,KumagaiBook}, with very general results to this effect established in the recent work of Lee \cite{lee2020relations}.
This theory has led to a fairly complete understanding of several notable motivating examples including random planar maps \cite{MR4255139,MR4146545,MR3010812,curien2020geometric}, high-dimensional percolation and branching random walks \cite{MR2551766,BJKS08,MR3947331}, and uniform spanning trees in two  dimensions \cite{barlow2014subsequential}, three dimensions \cite{MR4348685},  and  high dimensions ($d>4$) \cite{MR4055195}. The analysis of other important examples such as two-dimensional critical percolation remain largely open despite significant partial progress \cite{Kesten86,kestensubdiff,MR4369717}.

As suggested by this list of examples, many of the most interesting random fractals arise from critical statistical mechanics models, and for many such models the geometric and spectral properties of the  associated random fractal depends heavily on the dimension in which the model is considered.
Indeed, for many random fractals arising in statistical mechanics,
  a dichotomy emerges around an \textit{upper-critical dimension}~\cite{10.1093/acprof:oso/9780198509233.001.0001}, denoted $d_c$, which is equal to $4$ for the uniform spanning tree and $6$ for percolation:  below this dimension, the behaviour of the fractal is highly dependent on the geometry of the underlying space, while above this dimension  the fractal displays \emph{mean-field behaviour}, meaning that its large-scale behaviour is the same as it would be in a `geometrically trivial' setting such as the complete graph or the binary tree. For many models the mean-field regime is described by \emph{Alexander-Orbach} behaviour \cite{alexander1982density,MR2247823,kestensubdiff}, in which the relevant random fractal has quadratic volume growth, spectral dimension $4/3$, and typical $n$-step walk displacement of order~$n^{1/3}$. Indeed, Alexander-Orbach behaviour has been proven to hold for high-dimensional \emph{oriented} percolation by Barlow, Jarai, Kumagai, and Slade \cite{BJKS08}, high-dimensional percolation by Kozma and Nachmias \cite{MR2551766}, and for the high-dimensional uniform spanning tree by the second author \cite{MR4055195}. (An interesting example that is \emph{not} expected to exhibit Alexander-Orbach behaviour in high dimensions is the minimal spanning forest, mean-field models of which have cubic volume growth and spectral dimension $3/2$ \cite{addario2013local,nachmias2022wired}.)

    At the upper-critical dimension itself ($d=d_c$), it is expected that mean-field behaviour \emph{almost} holds, with many quantities of interest expected to exhibit a polylogarithmic correction to their mean-field scaling. It is this regime that provides the focus of this paper, in which we determine the precise order of the polylogarithmic corrections to scaling for the geometric and spectral properties of the uniform spanning tree (UST) at its upper-critical dimension $d_c=4$. 
    The particular polylogarithmic corrections we compute are those governing the volume of balls, the resistance across them, and the return probabilities, range, displacement 
     and exit times of random walks on the tree. Most of our  work goes into estimating the volume growth and resistance growth of the 4d UST, with the associated random walk estimates following straightforwardly by techniques developed in \cite{BJKS08,kumagai2008heat} that are by now rather standard. (The relevant proofs are presented in a self-contained way in Section~\ref{subsec:Kumagai_Misumi}.)
 We believe that this is the first time that polylogarithmic corrections to Alexander-Orbach behaviour have been computed for the random walk on a random fractal at the upper-critical dimension. Following \cite{MR2925573}, which computes the exact polylogarithmic corrections to a random walk on the four-dimensional random walk trace, we also believe that our work is the second
   time such polylogarithmic corrections to random walk behaviour at the upper critical dimension have been computed for any model. Partial progress on this problem for other models includes \cite{MR4452652} (see also \cite{MR3231996}) in which the existence of a non-trivial polylogarithmic correction to resistance growth is established for oriented branching random walk in $\Z^6\times\Z_+$.

\medskip

\subsection{The uniform spanning tree}

Over the last thirty years, the uniform spanning tree has emerged as a model of central importance throughout probability theory, with close connections to many other topics including electrical networks \cite{Kirch1847,BurPe93},  loop-erased random walk \cite{Lawler80,Wilson96,BLPS}, the dimer model \cite{Ken00,berestycki2020dimers}, the Abelian sandpile model \cite{JarRed08,JarWer14,MajDhar92,bhupatiraju2016inequalities,MR4055195} and the random cluster model \cite{GrimFKbook,Hagg95}. 
 Aside from these connections, the UST is also interesting as an example of a model exhibiting much of the rich phenomena associated to critical statistical mechanics models, but which is much more tractable to study than essentially any other (non-Gaussian) model thanks to its close connection to random walks via Wilson's algorithm \cite{Wilson96,BLPS} and the Aldous-Broder algorithm \cite{Aldous90,broder1989generating,hutchcroft2015interlacements}.

We now very briefly introduce the model, referring the reader to e.g.\ \cite{LP:book,barlow2014loop,MR4055195} for further background.
 The \textbf{uniform spanning tree} of a finite connected graph is defined by choosing a spanning tree (i.e.\ a connected subgraph that contains every vertex and no cycles) of the graph uniformly at random. Pemantle~\cite{Pem91} proved that there is a well-defined infinite volume limit of the uniform spanning tree of the hypercubic lattice $\Z^d$ which does not depend on the boundary conditions used when taking the limit and which is connected a.s.\ if and only if $d\leq 4$ (see also \cite{BLPS}). This infinite volume limit is known as the \textbf{uniform spanning tree} of $\Z^d$ when $d\leq 4$ and the \textbf{uniform spanning forest} of $\Z^d$ when $d\geq 5$. 
The critical dimension $d=4$ is characterized by the UST just barely managing to be connected, with 
two points at Euclidean distance $n$ typically connected by a path of Euclidean diameter much\footnote{Heuristic calculations suggest that the path connecting two distant points $x$ and $y$ has Euclidean diameter distributed approximately like $\|x-y\|^{1+Z}$ where $Z$ is an exponential random variable.} larger than $n$ and with the length of the path in the tree connecting two neighbouring vertices having an extremely heavy $(\log n)^{-1/3}$ tail \cite{MR1364896}. This heavy tail on the probability of an abnormally long connection, and the related fact that the length of a loop-erased random walk in four dimensions is only very weakly concentrated, is responsible for much of the technical difficulties encountered in the paper. For example, it makes it difficult to justify the important heuristic that the volume of the intrinsic $n$-ball in the tree comes mostly from `typical' points for which the tree-geodesic to the origin has Euclidean diameter of order $n^{1/2} (\log n)^{1/6}$.

\subsection{Distributional asymptotic notation}
To facilitate a clean presentation of our main results, we use \emph{distributional asymptotic notation} (a.k.a.\ ``big-O and little-o in probability" notation). Since this notation is not at all standard in probability theory\footnote{Indeed, it is surprising how non-standard this notation is in the probability theory literature given how useful it is. The use of this notation has previously been advocated by Janson \cite{https://doi.org/10.48550/arxiv.1108.3924} who uses the notation $O_p,o_p,\Theta_p$ etc.\ rather than $\bO,\bo,\bTheta$  as we write here. We use bold letters rather than $p$ subscripts since e.g.\ $O_p$ would typically be used in probability to denote a deterministic upper bound whose implicit constants depend on a parameter $p$, and we wish to avoid clashing with this existing notational convention. The particular choice to use bold characters was made since LaTeX includes bold fonts for Greek characters by default while e.g.\ $\backslash$mathscr$\{\backslash$Theta$\}$ and $\backslash$mathcal$\{\backslash$Theta$\}$ are not defined.
}, let us take a moment to explain how it is used. We hope the reader will find this diversion worthwhile after seeing how clean the statements of our main theorems are compared with similar results in the literature, and consider using this notation in their own work.


Before introducing this notation, let us first briefly introduce standard (deterministic) asymptotic notation as we use it. We write $\asymp$, $\succeq$, and $\preceq$ for equalities and inequalities holding to within positive multiplicative constants, so that if $f$ and $g$ are non-negative then ``$f(n) \preceq g(n)$ for every $n\geq 1$" means that there exists a positive constant $C$ such that $f(n)\leq Cg(n)$ for every $n\geq 1$. (We will often drop the ``for every $n\geq 1$'' and write simply ``$f(n)\preceq g(n)$'' when doing so does not cause confusion.) We use Landau's asymptotic notation similarly, so that $f(n)=O(g(n))$, $f(n)=\Omega(g(n))$, and $f(n)=\Theta(g(n))$ mean the same thing as $f(n) \preceq g(n)$, $f(n) \succeq g(n)$, and $f(n) \asymp g(n)$ respectively, while $f(n)=o(g(n))$ means that $f(n)/g(n)\to 0$ as $n\to\infty$. More complicated expressions can be obtained by putting this notation inside functions, so that e.g.\ $f(n)=O(e^{n-o(n^{1/2})})$ means that there exists a non-negative function $h(n)$ with $n^{-1/2}h(n)\to 0$ and a positive constant $C$ such that $f(n)\leq Ce^{n-h(n)}$ for every $n\geq 1$. Implicit constants and functions given by this notation will always be non-negative, and we denote quantities of uncertain sign using $\pm O$, $\pm o$, etc. (While this is not completely standard, it greatly increases the expressive power of the notation.) Be careful to note that when forming such compound expressions, $\Theta$ should always be interpreted as the conjunction of $O$ and $\Omega$, so that ``$f(n)=\Theta(e^{n-o(n)})$" means the same thing as ``$f(n)=O(e^{n-o(n)})$ and $f(n)=\Omega(e^{n-o(n)})$", which means that there exist positive constants $c$ and $C$ and \emph{possibly distinct} non-negative functions $h^+$ and $h^-$ with $\lim_{n\to\infty}n^{-1}h^+(n)=\lim_{n\to\infty} n^{-1}h^-(n) =0$ such that $f(n) \leq C e^{n-h^+(n)}$ and $f(n)\geq c e^{n-h^-(n)}$. Whenever we use asymptotic notation, we can add a qualifier such as ``as $n\to\infty$" to mean that the inequalities in question hold only for sufficiently large $n$; this will typically be used to avoid worrying about expressions such as $\log \log n$ being undefined or negative for small values of $n$.

 We use boldface characters to apply this notation in settings where the relevant bounds are guaranteed only to hold with high probability, rather than deterministically. Given two sequences of (possibly deterministic) non-negative random variables $(X_n)$ and $(Y_n)$ defined on the same probability space, we write
\begin{align*}
	X_n &= \bO(Y_n) &&\text{ to mean that }& \qquad \lim_{\lambda \to \infty} \sup_n \mathbb{P}(X_n \geq \lambda Y_n) &= 0,\\
	X_n &= \bOmega(Y_n) &&\text{ to mean that }& \qquad \lim_{\lambda\to \infty} \sup_n \mathbb{P}(X_n \leq \lambda^{-1} Y_n) &= 0,\\
	X_n &= \bTheta(Y_n) &&\text{ to mean that }& \qquad X_n = \bO(Y_n) \text{ and } X_n &= \bOmega(Y_n), \text{ and} \phantom{\lim_n}\\
	X_n &= \bo(Y_n) &&\text{ to mean that }& \qquad \lim_{n\to \infty} \mathbb{P}(X_n \geq \eps Y_n) &= 0 \text{ for every $\eps>0$}.
\end{align*}
In other words, $X_n=\bO(Y_n)$ and $Y_n=\bOmega(X_n)$ both mean that $\{X_n/Y_n\}$ is tight in $[0,\infty)$, $X_n=\bTheta(Y_n)$ means that $\{X_n/Y_n\}$ is tight in $(0,\infty)$, and $X_n=\bo(Y_n)$ means that $X_n/Y_n$ converges to zero in probability.
As in the deterministic case, we can add a qualifier ``as $n\to\infty$'' to mean that there exists $n_0<\infty$ such that the relevant inequalities hold between $X_n$ and $Y_n$ provided that $n\geq n_0$. Let us stress again that, as in the deterministic case, the random variables denoted implicitly by our use of asymptotic notation are always taken to be non-negative. When we wish to apply this notation to quantities of uncertain sign we use $\pm \bO$, $\pm \bo$, etc.\ as appropriate.


Like in the deterministic case, this notation really begins to shine when forming more complicated compound expressions. Again, we warn the reader that in such an expression, the implicit random variables (e.g.\ those appearing in an exponent) may be different in the upper and lower bounds. Indeed this will usually be the case in our applications. 
To give a contrived example in which all these conventions come into force,  ``$X_n = \bTheta (\exp[n+\bO((\log n)^{\bO(1)})\pm\bo(\log \log n)])$ as $n\to \infty$" is equivalent to the statement that there exists $n_0<\infty$ and sequences of non-negative  random variables $(A_n^-)$, $(A_n^+)$, $(B_n^-)$, $(B_n^+)$, $(C_n^-)$, and $(C_n^+)$ and real-valued sequences of random variables $(D_n^-)$ and $(D_n^+)$ such that $(A_n^-)$ is tight in $(0,\infty]$, $(A_n^+)$, $(B_n^-)$, $(B_n^+)$, $(C_n^-)$, and $(C_n^+)$ are tight in $[0,\infty)$, $(D_n^-)$ and $(D_n^+)$ converge to zero in probability, and
\[
A_n^- e^{n+B_n^- (\log n)^{C_n^-}+D_n^- \log \log n} \leq X_n \leq A_n^+ e^{n+B_n^+ (\log n)^{C_n^+}+D_n^+ \log \log n} \qquad \text{ for every $n\geq n_0$}.
\]
Note the incredible economy we have achieved by writing this complicated condition in the simple form ``$X_n = \bTheta (\exp[n+\bO((\log n)^{\bO(1)})\pm\bo(\log \log n)])$ as $n\to \infty$"!
\begin{remark}
	As with deterministic asymptotic notation, there are many useful elementary notational identities. Of these, we will repeatedly use that for any sequence of random variables $(X_n)_{n\geq0}$ if $X_n=\mathbf{o}(Y_n)$ then $X_n=\mathbf{O}(Y_n)$, and if $X_n=\mathbf{O}(Y_n(\log n)^\delta)$ for all $\delta>0$, then $X_n=\mathbf{O}(Y_n(\log n)^{o(1)})$. Similarly, if $X_n=\mathbf{\Omega}(Y_n(\log n)^{-\delta})$ for all $\delta>0$, then $X_n=\mathbf{\Omega}(Y_n(\log n)^{-o(1)})$.
\end{remark}

\subsection{Statement of results}

We now state our main results. We begin with our results on the volumes of intrinsic balls, the proof of which occupies the majority of the paper.

\begin{theorem}[Volume growth] \label{thm:main_geometric_theorem}
	Let $\fT$ be the uniform spanning tree of $\Z^4$ and for each $n\geq 0$ let $\fB(n)=\fB(0,n)$ denote the intrinsic ball of radius $n$ around the origin in $\fT$. The volume of $\fB(n)$ satisfies the distributional asymptotics
	\[
	|\mathfrak{B}(n)| = \mathbf{\Theta}\left(\frac{n^2}{(\log n)^{1/3-o(1)}}\right) \qquad \text{ and } \qquad \mathbb{E}\abs{\mathfrak{B}(n)} = \Theta\left(\frac{n^2}{(\log n)^{1/3-o(1)}}\right)
	\]
	as $n\to\infty$. Moreover, letting $\Lambda(r)$ denote the $\ell^\infty$ ball of radius $r$ around the origin in $\Z^4$ for each $r\geq 0$, we have that
	\[
	\lim_{n\to\infty}\P\left(\mathfrak{B}(n)\subseteq \Lambda\Bigl(n^{1/2}(\log n)^{1/6+\delta}\Bigr)\right) =1
	\] 
  for every $\delta>0$.
\end{theorem}
Recall that in high dimensions the components of the uniform spanning forest have quadratic volume growth $\abs{\fB(n)}=\mathbf{\Theta}(n^2)$ \cite{MR4055195,barlow2016geometry}, so that the behaviour in four dimensions differs from the high-dimensional behaviour by a polylogarithmic factor as expected.

The proofs of both the upper and lower bounds of Theorem~\ref{thm:main_geometric_theorem} rely on Wilson's algorithm \cite{Wilson96,BLPS} to express properties of the tree in terms of properties of loop-erased random walks. Accordingly, they also both rely on an understanding of the behaviour of the loop-erased random walk in four dimensions developed  in \cite{MR1117680,MR1364896,MR4038044}, with the proof of the lower bounds
also relying on the control of the \emph{capacity} of the loop-erased walk developed in \cite{hutchcroft2020logarithmic,MR3945751}.
The proof of the upper bound also uses a generalisation of the method of \emph{typical times} introduced in \cite{hutchcroft2020logarithmic}, a very useful technical tool that allows us to circumvent several issues that arise from the fact that the length of a four-dimensional loop-erased random walk is only very weakly concentrated. (The use of this machinery is also responsible for the presumably unnecessary subpolylogarithmic $(\log n)^{\pm o(1)}$ errors appearing throughout our results.)

We now turn to our results concerning the \emph{random walk} on the four-dimensional UST. We write $\mathbb{P}$ and $\mathbb{E}$ for probabilities and expectations taken with respect to the joint law of the UST $\fT$ on $\Z^4$ and the random walk $X=(X_n)_{n\geq 0}$ on $\fT$ started at the origin, and write $\bP^\fT$ and $\bE^\fT$ for probabilities and expectations taken with respect to the conditional law of $X$ given $\fT$.
We write $p^{\mathfrak{T}}_n(x,y)$ for the transition probabilities of a random walk on the uniform spanning tree $\mathfrak{T}$ conditional on $\mathfrak{T}$, write $\tau_n$ for the time taken for the random walk to hit the complement of the intrinsic ball of radius of $n$, and write $d_\fT$ for the intrinsic distance 
 on $\fT$.

		\begin{theorem}[Random walk asymptotics]\label{thm:est_collect}
	Let $\fT$ be the uniform spanning tree of $\Z^4$ and let $X=(X_n)_{n\geq 0}$ be the simple random walk on $\fT$ started at the origin. The following distributional asymptotic expressions hold as $n\to\infty$:
	\begin{align}
		&&\emph{Intrinsic displacement}:&&\hspace{0.5cm}
		\,d_{\fT}(X_0,X_n),\,\max_{0\leq i\leq n} d_{\fT}(X_0,X_i) &= \mathbf{\Theta}\Biggl(\hspace{0.115em}n^{\frac{1}{3}}\hspace{0.115em}(\log n)^{\frac19- o(1)} \Biggr)
		\label{eq:distance_intr_lower} \hspace{1cm}\\
		&&\emph{Extrinsic displacement}:&& 
     \max_{0\leq i\leq n} \|X_i\|_{\infty}&=\mathbf{\Theta}\Biggl(\hspace{0.115em}n^{\frac{1}{6}}\hspace{0.115em}(\log n)^{\frac29+ o(1)} \Biggr)			\label{eq:supdistance_extr_lower}\\
		&&\emph{Return probabilities}:&& p_{2n}^{\fT}(0,0)&=\mathbf{\Theta}\Biggl(\frac{1}{n^{\frac23}}(\log n)^{\frac19- o(1)}\Biggr)
		\label{eq:return_intr_upper}\\
			&&\emph{Range}:&&
		\#\{X_m:0\leq m \leq n\}&= \mathbf{\Theta}\Biggl(\hspace{0.115em}n^{\frac{2}{3}}\frac{1}{(\log 	n)^{\frac19\pm o(1)}} \hspace{-0.125em}\Biggr)\\
		&&\emph{Hitting times}:&& \tau_n,\,  \bE^\fT[\tau_n]&= \mathbf{\Theta}\Biggl(\hspace{0.085em}n^3\hspace{0.085em} \frac{1}{(\log n)^{\frac13-  o(1)}}\Biggr). \label{eq:exit_intr}
	\end{align}
\end{theorem}


\begin{remark}
It is reasonably straightforward to adapt the proofs of \cite{MR4055195} to prove that, in four dimensions, all the quantities we consider here satisfy Alexander-Orbach asymptotics up to $(\log n)^{\pm O(1)}$ factors. Identifying the \emph{correct} powers of $\log$ is significantly more difficult and is the primary contribution of this paper.
\end{remark}

 As mentioned above, the behaviour of the random walk on the uniform spanning tree has previously been studied in dimensions $d=2$ \cite{barlow2011spectral,barlow2014subsequential} $d=3$ \cite{MR4348685}, and $d\geq5$ \cite{MR4055195}, with the two cases $d=2$ and $d=3$ presenting unique challenges that are largely distinct from those associated to the critical dimension $d=d_c=4$ considered here.
While we are the first to study the polylogarithmic corrections to the volume of balls and the behaviour of random walks on the UST at $d=4$, our work builds upon the substantial literature studying other aspects of the 4d UST, the highlights of which include 
 \cite{MR1364896,schweinsberg2009loop,hutchcroft2020logarithmic,MR4038044,MR1117680,MR835809}. Our work is influenced most strongly by the recent work of Sousi and the second author \cite{hutchcroft2020logarithmic}; we rely on both the results proven and the techniques developed in that paper in numerous ways.

\medskip

Following
  Kumagai-Misumi \cite{kumagai2008heat}, which collects and generalises results of \cite{MR2177164,BJKS08,MR2247823}, estimates of the form proven in Theorem~\ref{thm:est_collect} can all be deduced from the volume growth estimates of Theorem~\ref{thm:main_geometric_theorem} together with estimates on the \emph{effective resistance} between the origin and the boundary of a ball in the tree. The relevant effective resistance estimates will in turn be deduced from Theorem~\ref{thm:main_geometric_theorem} together with the asymptotics of the \emph{intrinsic arm probability} computed in \cite{hutchcroft2020logarithmic}.
We let $\mathscr{R}_{\mathrm{eff}}(A\leftrightarrow B;G)$ denote the effective resistance between sets $A,B\subseteq V[G]$ in the graph $G$, where we assign unit resistance to each edge $e\in E[G]$, so that if $\deg_\fT(0)$ denotes the degree of $0$ in $\fT$ then $\mathscr{R}_{\mathrm{eff}}(0\leftrightarrow \partial \fB(0,n);\fT)^{-1}:=\deg_\fT(0) \bP^\fT($hit $\partial \fB(0,n)$ before returning to~$0).$ Background on effective resistances can be found in e.g.\ \cite{LP:book,KumagaiBook}.

\begin{theorem}[Effective resistance] \label{cor:eff_cond}
	Let $\fT$ be the uniform spanning tree of $\Z^4$ and for each $n\geq 0$ let $\partial \fB(n)= \partial \fB(0,n)$ denote the set of vertices with distance exactly $n$ from the the origin in $\fT$. Then
	\[\mathscr{R}_{\mathrm{eff}}(0\leftrightarrow \partial \fB(0,n);\fT) = n (\log n)^{-\mathbf{o}(1)}\]
	as $n\to\infty$.
\end{theorem}

Note that the linear upper bound $\mathscr{R}_{\mathrm{eff}}(0\leftrightarrow \partial \fB(0,n)) \leq n$ is trivial and holds for any graph. Together with existing results in other dimensions \cite{barlow2014subsequential,MR4348685,MR4055195}, Theorem~\ref{cor:eff_cond} shows that the UST has (approximately) linear effective resistance growth in \emph{every} dimension.
As will be clear from the proof, this is a consequence of the scaling relation
\begin{equation}
\label{eq:scaling_relation}
\P(\text{the past of the origin has intrinsic diameter $\geq n$}) \approx \frac{n}{\text{typical volume of an intrinsic $n$-ball}},
\end{equation}
which also holds in every dimension. Here, the \textbf{past} of the origin 
is the union of the origin and the finite connected component of the UST left when the origin is deleted; estimating the probability that the past is large in various  senses is the main subject of \cite{hutchcroft2020logarithmic}, which in particular establishes up-to-constants estimates on the left hand side of \eqref{eq:scaling_relation}.
  Currently, however, there is no direct proof of this scaling relation, which in four dimensions is verified only by computing the two sides separately in \cite{hutchcroft2020logarithmic} and the present paper. It would be very interesting to have a direct and general proof of this relation in all dimensions that worked without computing either quantity.

\medskip

While Theorems~\ref{thm:main_geometric_theorem} and \ref{cor:eff_cond} are sufficient to compute the exact logarithmic corrections to the asymptotic properties of the random walk on the UST 
using the methods of \cite{kumagai2008heat,barlow2016geometry} as discussed above,
we will also show that a significantly stronger bound on the displacement of the random walk
   can be proven using the Markov-type method pioneered in the work of Lee and coauthors \cite{lee2020relations,MR4348678,MR4369717,MR3077911}.
\begin{theorem}[Sharp upper bounds on the mean-squared displacement] \label{thm:displacement}
	Let $\fT$ be the uniform spanning tree of $\Z^4$ and let $X=(X_n)_{n\geq 0}$ be the simple random walk on $\fT$ started at the origin. Then
	\begin{equation}
		\E{\max_{0\leq i \leq n}d_{\fT}(X_0,X_i)^2}\preceq n^{2/3}(\log n)^{2/9}
	\end{equation}
	for every $n\geq 2$.
\end{theorem}

The specific argument used to prove this theorem is inspired closely by the work of Ganguly and Lee~\cite{MR4369717}.
Briefly, the idea is to use the universal Markov-type inequality for weighted metrics on trees \cite{MR3077911} to prove a \emph{diffusive} upper bound for the random walk with respect to a modified metric supported only on vertices of the tree whose past has large intrinsic diameter, then deduce the desired subdiffusive estimate in the original metric. The tail bounds on the intrinsic diameter of the past of the origin proven in \cite{hutchcroft2020logarithmic} are precisely what is needed to carry this argument through. In particular, the proof of Theorem~\ref{thm:displacement} does \emph{not} rely on Theorem~\ref{thm:main_geometric_theorem} or the theory of typical times, allowing us to avoid the $(\log n)^{\pm o(1)}$ present in the statement of that theorem. 

\begin{remark}As discussed in \cite[Example 2.6]{BJKS08}, although the \emph{typical} displacement of the random walk can always be controlled in terms of volume growth and resistance growth, it is possible in general for the displacement not to be uniformly integrable, so that its mean grows significantly faster than its median. As such, the second moment estimate provided by Theorem~\ref{thm:displacement} is significantly stronger than what can be deduced directly from Theorems~\ref{thm:main_geometric_theorem} and \ref{cor:eff_cond} by the techniques of \cite{kumagai2008heat,BJKS08}.
\end{remark}

\section{Intrinsic volume growth}
\label{sec:volume}
In this section we prove Theorem~\ref{thm:main_geometric_theorem}. The upper and lower bounds of the theorem, which use completely different techniques, are proven in Sections~\ref{subsec:volume_upper} and \ref{subsec:volume_lower} respectively.
Both parts of the proof will utilize the connections between the uniform spanning tree and the loop-erased random walk implied by Wilson's algorithm, and so to proceed we must provide notation for the loop-erased random walk and some related quantities.\\

\noindent \textbf{Loop-erased random walk. }
For each $-\infty\leq n\leq m\leq \infty$, let $L(n,m)$ be the graph with vertex set $\{i\in\Z:n\leq i\leq m\}$ and edge set $\{\{i,i+1\}:n\leq i\leq m-1\}$. A path is then a multigraph homomorphism from $L(n,m)$ to the hypercubic lattice $\Z^4$ for some $-\infty\leq n\leq m\leq \infty$. We write $w_i=w(i)$ for the vertex visited at time $i$.  For $n\leq b\leq m$, we write $w^b$ for the restriction of $w$ to $[n,b]$, and call $w^b$ the path \textbf{stopped at} $b$. In particular, given a random walk $X$, we will often use the notation $X^T$ for a random walk stopped at some possibly random time $T$. A path is said to be \textbf{transient} if it visits every vertex of $\Z^4$ at most finitely many times. In particular, finite paths are always transient. Given a transient path $w:L(0,m)\rightarrow\Z^4$, we recursively define the sequence of times $\ell_n(w)$ by $\ell_0(w)=0$, and
\[
\ell_{n+1}(w) = 1 + \max\{k:w_k=w_{\ell_n}\},
\]
where we terminate the sequence the first time $\max\{k:w_k=w_{\ell_n}\}=m$ when $m<\infty$. The loop-erasure of $w$ is then the path induced by the sequence of neighbouring vertices
\[
\LE(w)_i = w_{\ell_i(w)}.
\]
We will also need the quantity
\[
\rho_n(w) = \max\{m\geq 0:\ell_m(w)\leq n\},
\]
which for each $n\geq 0$  counts the number of points up to time $n$ (excluding $w_0$) which are not erased when computing the loop-erasure of $w$, so that $(\ell_n)_{n\geq 0}$ and $(\rho_n)_{n\geq 0}$ are inverses of each other in the sense that
\[
\ell_n(w)\leq m \quad \text{if and only if} \quad \rho_m(w) \geq n,
\]
for every $n,m\geq 0$. 

The loop-erasure of a simple random walk is known as the \textbf{loop-erased random walk}. The theory of loop-erased random walk was both introduced and developed extensively by Lawler \cite{Lawler80}, whose results on the four-dimensional loop-erased random walk \cite{MR1364896,MR1117680}
 play an extensive role in this paper both directly and through inputs to \cite{hutchcroft2020logarithmic}. Given a random walk $X$, we will usually abbreviate $\ell_n=\ell_n(X)$ and $\rho_n=\rho_n(X)$.
It will also be convenient to define the notation
\[
\LE_\infty(X^n) := \LE(X)^{\rho_n}
\]
for $n\geq 0$, giving the component of the infinite loop erasure $\LE(X)$ which is contributed by the first $n$ steps of the random walk $X$. We emphasise that the brackets of $\LE_\infty(X^n)$ do \textit{not} indicate that $\LE_\infty(X^n)$ is a function of just $X^n$.
The following concentration estimates of Lawler \cite{MR1117680,MR1364896}, as stated in \cite[Theorem 2.2]{hutchcroft2020logarithmic}, will be used repeatedly throughout the the paper.
\begin{theorem}[\!\!\cite{hutchcroft2020logarithmic}, Theorem 2.2]\label{theorem:LERWconc} Let $X$ be a simple random walk on $\Z^4$, then
  \begin{align*}\pb\left(\left\vert\frac{\rho_n}{n(\log n)^{-1/3}}-1\right\vert>\varepsilon\right)&\preceq_\varepsilon \frac{\log\log n}{(\log n)^{2/3}}\qquad\text{and hence}\\
    \pb\left(\left\vert\frac{\ell_n}{n(\log n)^{1/3}}-1\right\vert>\varepsilon\right)&\preceq_\varepsilon \frac{\log\log n}{(\log n)^{2/3}},
  \end{align*}
  for every $\epsilon>0$ and $n\geq 3$.
\end{theorem}

\textbf{Wilson's algorithm rooted at infinity} \cite{Wilson96,BLPS} allows us to build a sample of the UST of $\Z^4$ (or any other transient graph) out of loop-erased random walks. This algorithm is very important to most analyses of the UST. We will assume that the reader is already familiar with Wilson's algorithm, referring them to e.g.\ \cite{LP:book} for background otherwise.

\medskip

Finally, let us introduce notation concerning the geometry of $\Z^4$ and the tree $\fT$.
  We write $\norm{x}$ for the $\ell^\infty$ norm of $x\in\Z^4$ and write $\Lambda(x,r)$ for the $\ell^\infty$ ball around $x\in\Z^d$ of radius $r$. For convenience, we will write $\Lambda(r)$ for $\Lambda(0,r)$. For each $x\in \Z^4$ and $r\geq 1$, $\mathfrak{B}(x,r)$ will denote the intrinsic ball of radius $r$  around $x$ in $\fT$, with $\mathfrak{B}(r):=\mathfrak{B}(0,r)$. 
For each pair of vertices $x,y\in \Z^4$ we write $\Gamma(x,y)$ for the unique simple path 
 between $x$ and $y$ in 
$\fT$, which is well-defined since the UST of $\Z^4$ is a.s.\ connected \cite{Pem91,BLPS}, and write $\Gamma(x,\infty)$ for the \textbf{future} of $x$ in $\mathfrak{T}$, i.e.\ the unique infinite simple path in $\mathfrak{T}$ with $x$ as an endpoint, which is well-defined since the UST of $\Z^4$ is one-ended a.s.\ \cite{Pem91,BLPS}. Given two vertices $x,y\in\Z^4$ we will denote by $x\vee y=y\vee x$ the unique point at which the futures of $x$ and $y$ in $\mathfrak{T}$ first intersect.

The \textbf{past} of a vertex $v$ in the uniform spanning tree $\fT$, denoted\footnote{This character is $\backslash\mathrm{mathfrak}\{\mathrm{P}\}$.} $\mathfrak{P}(v)$, is the union of the vertex and the finite components that are disconnected from infinity when the vertex is deleted from $\fT$. We write $\mathfrak{P}(v,n)$ for $\mathfrak{P}(v)\cap \mathfrak{B}(v,n)$ 
 and write $\partial\mathfrak{B}(v,n)$ for the set of vertices in $\mathfrak{T}$ at intrinsic distance exactly $n$ from $v$.
 Further discussion of the basic topological features of the UST used here can be found in \cite[Chapter 10]{LP:book}.

\subsection{Upper bounds}
\label{subsec:volume_upper}

In this section we prove the following two propositions,  which establish the upper bounds of Theorem \ref{thm:main_geometric_theorem}.
Throughout this section we will write $\asymp$, $\preceq$, and $\succeq$ with subscripts such as $\delta$ and $p$ to mean that the implicit constants are allowed to depend on these parameters.

\begin{proposition} \label{prop:volume_upper}
	Let $\fT$ be the uniform spanning tree of $\Z^4$. Then 
	\[
	\mathbb{E}\abs{\mathfrak{B}(n)} = O\left(\frac{n^2}{(\log n)^{1/3-o(1)}}\right) 
	\]
  as $n\to\infty$.
\end{proposition}

\begin{proposition}\label{prop:extr_containment}
	Let $\fT$ be the uniform spanning tree of $\Z^4$ and let $\delta>0$. Then
	\[
	\pr\left(
  \fB(n) \nsubseteq 
  \Lambda \Bigl(n^{1/2}(\log n)^{1/6+\delta}\Bigr)\right)\preceq_\delta \frac{\log\log n}{(\log n)^{2/3}}
	\]
	for every $n\geq 3$.
\end{proposition}

Both of these results will be proven using the following  supporting technical proposition, which bounds in expectation the amount of the volume of intrinsic balls which come from paths of atypical diameter.

\begin{proposition} \label{prop:subsume}
	Let $\fT$ be the uniform spanning tree of $\Z^4$, let $\delta>0$ and let $p\geq 1$. 
	Then 
	\[
	\mathbb{E}\,\big\vert\big\{x\in\mathfrak{B}(n):\Gamma(0,x)\nsubseteq  \Lambda\big(n^{1/2}(\log n)^{1/6+\delta}\big)\big\}\big\vert \preceq_{p,\delta} \frac{n^2}{(\log n)^{p}}
	\]
	for every $n\geq 2$.
\end{proposition}

The expected intrinsic volume bound of Proposition \ref{prop:volume_upper} follows immediately from Proposition~\ref{prop:subsume} together with \cite[Proposition 7.3]{hutchcroft2020logarithmic}, which provides a tight upper bound on the number of points connected to the origin inside an \emph{extrinsic} box of a given radius.

\begin{proposition}[\!\!\cite{hutchcroft2020logarithmic}, Proposition 7.3]\label{prop:extr_volume} Let $\fT$ be the uniform spanning tree of $\Z^4$. Then 
	\[
	\mathbb{E}\,\abs{\{x\in\Z^4: \Gamma(0,x)\subseteq \Lambda(r)\}}\preceq \frac{r^4}{\log r} 
	\]
  for every $r\geq 2$.
\end{proposition} 

\begin{proof}[Proof of Proposition \ref{prop:volume_upper}]
	Fix $\delta>0,\,p\geq 1$ and $n\geq 4$. We have trivially that
	\begin{align*}
		\mathbb{E}|\mathfrak{B}(n)|&\leq \mathbb{E}\,\big\vert\{x\in\mathfrak{B}(n):\Gamma(0,x)\not\subset\Lambda(n^{1/2}(\log n)^{1/6+\delta})\} \big\vert+\mathbb{E}\,\abs{\{x\in\Z^d: \Gamma(0,x)\subseteq \Lambda(n^{1/2}(\log n)^{1/6+\delta})\}}.
	\end{align*}
Applying Proposition \ref{prop:subsume} to the first term on the right hand side and Proposition \ref{prop:extr_volume} to the second yields that
	\begin{align*}
	\mathbb{E}|\mathfrak{B}(n)|\preceq_{p,\delta} \frac{n^2}{(\log n)^{1/3-4\delta}}+\frac{n^2}{(\log n)^{p}},
\end{align*}
which implies the claim since $\delta>0$ and $p\geq 1$ were arbitrary.
\end{proof}

The deduction of Proposition \ref{prop:extr_containment} from Proposition~\ref{prop:subsume} requires a more involved argument using further results of \cite{hutchcroft2020logarithmic} and is given after the proof of Proposition \ref{prop:subsume}.

\medskip

To prove Proposition~\ref{prop:subsume} we need to be able to relate balls in the extrinsic metric (i.e.\ the $\ell^\infty$ metric on $\Z^4$) to balls in the intrinsic metric. Intuitively, since paths in the UST are distributed as loop-erased random walks and since length-$n$ loop-erased random walks in $\Z^4$ are typically generated by simple random walks of length roughly $n (\log n)^{1/3}$ \cite{MR1364896}, we expect that intrinsic paths of length $n$ in the UST should have extrinsic diameter concentrated around $n^{1/2} (\log n)^{1/6}$. Unfortunately, however, the concentration estimates that are available for the length of loop-erased random walks are far too weak to directly rule out that most of the volume of the intrinsic ball comes from paths of atypically large diameter.
We circumvent this problem using a generalization of the \textit{typical time} methodology of \cite[Section 8]{hutchcroft2020logarithmic}, originally introduced to prove tail estimates on the extrinsic radius of the past of the origin: we will use typical times to subsume balls in the intrinsic metric by balls of an appropriate radius in the extrinsic metric.

\medskip

\textbf{Typical times.} We  now detail the generalised typical time methodology that we use. Given points $x,y\in\Z^4$ and a simple path $\gamma$ starting at $x$ and ending at $y$,  let $X$ be a random walk started at $x$ and conditioned to hit $y$ and to have loop erasure $\gamma$ when it first hits $y$. Roughly speaking, the typical time $T(\gamma)$ of $\gamma$ is defined to be the typical length of the walk $X$ under this conditional distribution; an important part of the theory is that this length is concentrated around the typical time $T(\gamma)$ under mild conditions on the path $\gamma$. 
Our proofs will apply a slight generalization of this notion, which we now introduce.
Instead of stopping the walk at a single point $y$, we introduce disjoint sets $A,B\subset\Z^d$ and define the \textbf{$(A,B)$-typical time} $T_{A,B}(\gamma)$ of a simple path $\gamma$ starting at $x$, ending when it first hits $A$, and avoiding $B$ to be
\[
T_{A,B}(\gamma):=\bE_x\left[\sum_{i=1}^{\abs{\gamma}} \Big(\ell_i(X^{\tau_A})-\ell_{i-1}(X^{\tau_A})\Big)\wedge \abs{\gamma}\ \Bigg\vert\ \tau_A<\infty, \tau_A<\tau_B, \LE(X^{\tau_A})=\gamma\right],
\]
where $\mathbf{E}_x$ denotes expectation with respect to the law of a simple random walk $X$ on $\Z^4$ started at $X_0=x$, and where the times $\ell_i(X^{\tau_A})$ are from the definition of the loop-erasure  of $X^{\tau_A}$, so that 
\[\tau_A= \ell_{|\gamma|}(X^{\tau_A})=\sum_{i=1}^{|\gamma|} \Big(\ell_i(X^{\tau_A})-\ell_{i-1}(X^{\tau_A})\Big)\]
when $\LE(X^{\tau_A})=\gamma$.
We will use boldface to denote probabilities and expectations taken with respect to the law of a simple random walk throughout the paper, so that $\pb_x$ will denote probability with respect to the law of a simple random walk started at time $0$ at vertex $x$. We remark that for paths $\gamma$ which hit $A$ and avoid $B$ we have that $T_{A,B}(\gamma)=T_{A\cup B,\emptyset}(\gamma)$, where we define $\tau_\emptyset = \infty$, and that the usual typical time as defined in \cite{hutchcroft2020logarithmic} is given by $T(\eta)=T_{\{\eta_n\},\emptyset}(\eta)$ when $\eta$ has length $n$.

\medskip

The following Lemma extends \cite[Lemma 8.2]{hutchcroft2020logarithmic} to $(A,B)$-typical times. The proof is identical to the proof of that lemma and is omitted.
\begin{lemma} \label{lemma:conc_T}
	There exists a constant $C$ such that if  $x\in\Z^4$,  $A,B$ are disjoint subsets of $\Z^4$, and $\gamma$ is a simple path of length $n\geq 0$ from $x$ to $A$ which does not intersect $B$, then
	\[
	\bP_x\left(\lvert\tau_A-T_{A,B}(\gamma)\rvert>\lambda n \ \Big|\ \tau_A<\infty, \tau_A<\tau_B, \LE(X^{\tau_A})=\gamma\right)\leq\frac{C}{\lambda},
	\]
	for every $\lambda\geq1$.
\end{lemma}

As explained in detail in \cite[Section 8]{hutchcroft2020logarithmic}, for most paths of interest the typical time $T(\gamma)$ is significantly larger than $|\gamma|$, so that Lemma~\ref{lemma:conc_T} can indeed be thought of as a concentration estimate, justifying the use of the `typical time' terminology. Indeed, when $\gamma$ is a loop-erased random walk of length $n$ its typical time will usually be of order $n (\log n)^{1/3}$. For an arbitrary path $\gamma$ of length $n\geq 1$ the best bounds are of the form
\begin{equation}
\label{eq:typical_time_extremal_bounds}
n \preceq T(\gamma) \preceq n \log (n+1);
\end{equation}
the lower bound is trivial while the upper bound follows by bounding the distribution of the length of the loop $\ell_i(X^{\tau_A})-\ell_{i-1}(X^{\tau_A})$ by that of the length of an unconditioned simple random walk loop in $\Z^4$ (see \cite[Equation (8.6)]{hutchcroft2020logarithmic}).
The upper bound is sharp when $\gamma$ is a straight line, while the lower bound is sharp when $\gamma$ is a space-filling curve.

As in \cite{hutchcroft2020logarithmic}, we bound typical times by a simpler functional that is easier to work with. If $\gamma$ has length $n$, we define $A_i(\gamma)=\sum_{k=1}^n \frac{1}{k}\mathrm{Esc}_k(\gamma^i)^2$, where  given a finite path $\eta$ of length $m$ and an integer $k\geq 1$ the $k$-step \textbf{escape probability} $\mathrm{Esc}_k(\eta)$ is defined by $\mathrm{Esc}_k(\eta)=\bP_{\eta_m}(X^k\cap \eta^{m-1}=\emptyset)$.
We then define \[\widetilde{T}(\gamma):=\sum_{i=0}^{n-1} A_i(\gamma).\]
It follows from the same calculations used to derive the analogous bound for the ordinary hitting time on \cite[Page 69]{hutchcroft2020logarithmic} that 
\begin{equation}
\label{eq:Ttilde_bounds_T}
\widetilde{T}(\gamma) \succeq T_{A,B}(\gamma)
\end{equation}
for every path $\gamma$ and every pair of disjoint sets $A,B \subseteq \Z^4$.
 For a given $0<\delta \leq 1$, we say that a finite path $\gamma$ of length $n\geq 0$ is \textbf{$\delta$-good} if 
\[
\sum_{i=0}^{n-1} A_{i}(\gamma)\mathbbm{1}\big(A_i\geq(\log n)^{1/3+\delta}\big)\leq \delta n,
\]
and say it is \textbf{$\delta$-bad} otherwise. If $\gamma$ is a $\delta$-good path of length $n\geq 2$, then 
\begin{equation}
\label{eq:delta_good_Ttilde_upper}
\widetilde{T}(\gamma)\leq\delta n +\sum_{i=0}^{n-1} A_i(\gamma) \mathbbm{1}\big(A_i<(\log n)^{1/3+\delta}\big)\preceq n(\log n)^{1/3+\delta}.
\end{equation}
We will apply \cite[Lemma 8.5]{hutchcroft2020logarithmic}, which is based on the work of Lawler \cite{MR1364896} (see also \cite{MR4038044}), and states that the loop erasure of a random walk is highly unlikely to be bad.

\begin{lemma}[\!\!\cite{hutchcroft2020logarithmic}, Lemma 8.5] \label{lemma:bad_bound} Let $\delta>0$ and $p\geq 0$ and let $X$ be simple random walk on $\Z^4$. Then
	\[
	\frac{1}{n} \sum_{k=0}^n \mathbf{P}_0\Big(\LE(X^k) \text{ is $\delta$-\text{bad}}\Big)\preceq_{\delta,p} \frac{1}{(\log n)^p},
	\]
	for every $n\geq 2$.
	
\end{lemma}

 We now apply this machinery to prove Proposition~\ref{prop:subsume}.
We will also use the mass-transport principle for $\Z^4$, which states that 
if 
$f:\Z^d\times\Z^d\rightarrow [0,\infty]$ is a diagonally invariant function, meaning that $f(x,y)=f(x+z,y+z)$ for every $x,y,z\in\Z^4$, then 
	$\sum_{x\in\Z^d}f(0,x)=\sum_{x\in\Z^d}f(-x,0)=\sum_{x\in\Z^d}f(x,0)$.

\begin{proof}[Proof of Proposition~\ref{prop:subsume}]	
	\noindent To prove the proposition, we will show that if $\mathscr{A}$ is any set of simple paths $\gamma$ with $\gamma_0=0$ and with length $\abs{\gamma}\leq n$, then
	\begin{multline}
		\sum_{v\in\Z^4} \pr\Bigl(\Gamma(0,v)\in \mathscr{A}, \text{ and }\Gamma(0,0 \wedge v)\subseteq \Lambda(0 \wedge v,r)\Bigr)
		\\ \preceq_{\delta,p}
		\sum_{v\in\Z^4} \pr\Bigl(\Gamma(0,v)\in \mathscr{A},\, \Gamma(0,0 \wedge v)\subseteq \Lambda(0 \wedge v,r)\\ \text{ and }\Gamma(v,0 \wedge v)\subseteq \Lambda(0 \wedge v,n^{1/2}(\log n)^{1/6+\delta}) \Bigr) + n^2 (\log n)^{-p}
		\label{eq:displacement_split_m}
	\end{multline}
	for every $0<\delta\leq 1$, $p\geq 1$, and $n,r\geq 2$. Before proving \eqref{eq:displacement_split_m}, let us first see how it implies the proposition. We must first define some notation. Given a finite path $\gamma=(\gamma_0,\ldots,\gamma_{\abs{\gamma}})$ and a vector $x$, we define $\gamma+x=(\gamma_0+x,\ldots,\gamma_{\abs{\gamma}}+x)$, and $\gamma^\leftarrow=(\gamma_{\abs{\gamma}},\ldots,\gamma_0)$. We extend these operations to sets of paths in the obvious way.
	Fix $\delta\in(0,1]$, $p\geq 1$ and define the two sets of paths 
	\begin{align*}\mathscr{A}_0&=\{\gamma :\gamma \text{ simple},\, \gamma_0 = 0,\, \abs{\gamma}\leq n,\, \gamma\nsubseteq \Lambda(0,\,\ \ \ n^{1/2}(\log n)^{1/6+2\delta})\}, \\	
	\mathscr{A}_0^\prime&=\{\gamma :\gamma \text{ simple},\, \gamma_0 = 0,\, \abs{\gamma}\leq n,\, \gamma\nsubseteq \Lambda(\gamma_{\abs{\gamma}},\,n^{1/2}(\log n)^{1/6+2\delta})\}.
	\end{align*}
 For any $x\in\Z^d$, writing $\mathscr{A}_0(x)$ for the set of paths $\mathscr{A}_0+x$, we observe that for any path $\gamma$ with $\gamma_0=x$, $\gamma_{\abs{\gamma}}=0$, we have that 
 \begin{equation} \label{eq:reverse_sets}
 	\gamma\in \mathscr{A}_0(x) \iff \gamma^\leftarrow\in \mathscr{A}_0^\prime.
 \end{equation}
		With this notation and observation in hand, setting $\mathscr{A}=\mathscr{A}_0$ in \eqref{eq:displacement_split_m} and taking $r\uparrow \infty$, we get
	\begin{align}
		&\mathbb{E}\,\big\vert\{x\in\Z^d:\Gamma(0,x)\in\mathscr{A}_0\} \big\vert=\sum_{v\in\Z^4} \pr\Bigl(\Gamma(0,v)\in \mathscr{A}_0\Bigr)
    \nonumber\\ 
		&\hspace{2cm}\preceq_{\delta,p}
		\sum_{v\in\Z^4} \pr\Bigl(\Gamma(0,v)\in \mathscr{A}_0(0) \text{ and }\Gamma(v,0 \wedge v)\subseteq \Lambda(0 \wedge v,n^{1/2}(\log n)^{1/6+\delta}) \Bigr) + n^2 (\log n)^{-p}
		\nonumber\\
		&\hspace{2cm}=_{\phantom{\delta,p}}
		\sum_{v\in\Z^4} \pr\Bigl(\Gamma(v,0)\in \mathscr{A}_0(v) \text{ and }\Gamma(0,0 \wedge v)\subseteq \Lambda(0 \wedge v,n^{1/2}(\log n)^{1/6+\delta}) \Bigr) + n^2 (\log n)^{-p}
		\nonumber\\
		&\hspace{2cm}=_{\phantom{\delta,p}}
		\sum_{v\in\Z^4} \pr\Bigl(\Gamma(0,v)\in \mathscr{A}_0^\prime \text{ and }\Gamma(0,0 \wedge v)\subseteq \Lambda(0 \wedge v,n^{1/2}(\log n)^{1/6+\delta}) \Bigr) + n^2 (\log n)^{-p}
	\end{align}
	for every $0<\delta \leq 1$, $p\geq 1$, and $n \geq 2$, where the second equality follows by an application of the mass-transport principle to exchange the roles of $0$ and $v$, and the third equality follows by \eqref{eq:reverse_sets}. Applying \eqref{eq:displacement_split_m} a second time with $\mathscr{A}= \mathscr{A}_0^\prime$ then yields that
	\begin{multline*}
		\mathbb{E}\,\big\vert\{x\in\Z^d:\Gamma(0,x)\in\mathscr{A}_0\} \big\vert \preceq_{\delta ,p}
		\sum_{v\in\Z^4} \pr\Bigl(\Gamma(0,v)\in \mathscr{A}_0^\prime,\, \\ \text{ and } \Gamma(0,0 \wedge v), \Gamma(v,0 \wedge v)\subseteq \Lambda(0 \wedge v,n^{1/2}(\log n)^{1/6+\delta}) \Bigr) + n^2 (\log n)^{-p},
	\end{multline*}
	and hence, applying the mass-transport principle a second time, we get
	\begin{align*}
		&\mathbb{E}\,\big\vert\{x\in\Z^d:\Gamma(0,x)\in\mathscr{A}_0\} \big\vert \phantom{\sum_{v\in\Z^4}}\\
		&\hspace{1cm}\preceq_{\delta ,p}
		\sum_{v\in\Z^4} \pr\Bigl(\Gamma(0,v)\in \mathscr{A}_0, \text{ and } \Gamma(0,0 \wedge v), \Gamma(v,0 \wedge v)\subseteq \Lambda(0 \wedge v,n^{1/2}(\log n)^{1/6+\delta}) \Bigr) + n^2 (\log n)^{-p}\\
		&\hspace{1cm}\preceq_{\delta\phantom{,p}}
		\sum_{v\in\Z^4} \pr\Bigl(\Gamma(0,v)\in \mathscr{A}_0, \text{ and } \Gamma(0,v)\subseteq \Lambda(2n^{1/2}(\log n)^{1/6+\delta}) \Bigr) + n^2 (\log n)^{-p}.
	\end{align*}
	If $n$ is sufficiently large that $(\log n)^\delta>2$ then the first term is zero and the claim follows.

	It remains to prove \eqref{eq:displacement_split_m}. Fix $0<\delta \leq 1,p\geq 1$ and $n,r \geq 2$. Let $\eta$ be the future of the origin in $\fT$ and write $\pr^\eta$ and $\mathbb{E}^\eta$ for probabilities and expectations taken with respect to the conditional law of $\fT$ given $\eta$. 
	Let $\mathcal{I}=\{i\in\{0,\ldots,n\}:\eta[0,i]\subseteq\Lambda(\eta_i,r)\}$, and for any $i\geq 0$ define the restriction $\mathscr{A}\vert_{x,\eta}^i$ to be the set of finite simple paths
	\[
	\mathscr{A}\vert_{x,\eta}^i = \{\gamma:\gamma_0=\eta_i,\, \gamma[1,\abs{\gamma}]\cap\eta=\emptyset,\,\gamma_{\abs{\gamma}}=x,\,\eta[0,i]\oplus\gamma[1,\abs{\gamma}]\in\mathscr{A}\},
	\]
	where for any two finite paths $\gamma$, $\gamma^\prime$, we have $(\gamma_0,\ldots,\gamma_{\abs{\gamma}})\oplus(\gamma_0^\prime,\ldots,\gamma_{\abs{\gamma^\prime}}^\prime)=(\gamma_0,\ldots,\gamma_{\abs{\gamma}},\gamma_0^\prime,\ldots,\gamma_{\abs{\gamma^\prime}}^\prime)$. In other words, $\mathscr{A}\vert_{x,\eta}^i$ is the set of simple paths (including paths of just a single vertex) beginning at $\eta_i$, avoiding the other points of $\eta$, and which when concatenated to $\eta[0,i-1]$ yield a path in $\mathscr{A}$ ending at $x$.
	
	For each $v\in\Z^d$ we can sample from the conditional distribution of the path in $\fT$ connecting $v$ to $\eta$ using Wilson's algorithm by starting a random walk $X$ at $v$ and loop erasing it when it first hits $\eta$. When sampling the path in this manner we have that the event $\{\Gamma(0,v)\in\mathscr{A} $ and $\Gamma(0,0\wedge v) \subseteq \Lambda(0\wedge v, r)\}$  occurs if and only if the union of disjoint events
	\[
	\bigcup_{i\in\mathcal{I}}\{\tau_i<\tau_{i}^c,\,\LE(X^{\tau_i})^\leftarrow\in\mathscr{A}\vert^i_{v,\eta}\}
	\]
	occurs, where we write $\tau_i$ for the hitting time of $\eta_i$ and write $\tau_{i}^c$ for the hitting time of $\eta\setminus \{\eta_i\}$, so that 
	\begin{equation}
		\label{eq:B(0,n)_SRW_expansion}
		\sum_{v\in\Z^4} \pr^\eta\Bigl(\Gamma(0,v)\in \mathscr{A}, \text{ and }\Gamma(0,0 \wedge v)\subseteq \Lambda(0 \wedge v,r)\Bigr)=  \sum_{i\in\mathcal{I}}\sum_{v\in \Z^4} \mathbf{P}_v\bigl(\tau_i<\tau_{i}^c,\,\LE(X^{\tau_i})^\leftarrow\in\mathscr{A}\vert^i_{v,\eta}\bigr).
	\end{equation}
	We remark that the probabilities on the right hand side of \eqref{eq:B(0,n)_SRW_expansion} are themselves random variables given that $\tau_i,\, \tau_i^c$ and $\mathscr{A}\vert^i_{v,\eta}$ depend on $\eta$. (The law of the simple random walk $\bP_v$ does not depend on $\eta$ and, since there is no possible ambiguity that $\mathbf{P}$ could denote expectation over the UST, we have chosen to made the dependence implicit.)
	
	Temporarily fixing $i\in\mathcal{I}$, we analyze the inner summation on the right hand side of \eqref{eq:B(0,n)_SRW_expansion} using the union bound
	\begin{multline} \label{eq:bad_split}
		\sum_{v\in\Z^4} \mathbf{P}_v\bigl(\tau_i<\tau_{i}^c,\,\LE(X^{\tau_i})^\leftarrow\in\mathscr{A}\vert^i_{v,\eta}\bigr) \leq  \sum_{v\in\Z^4} \mathbf{P}_v(\tau_i<\tau_{i}^c,\,\LE(X^{\tau_i})^\leftarrow\in\mathscr{A}\vert^i_{v,\eta},\, \LE(X^{\tau_i}) \text{ $\delta$-good}) \\+\sum_{v\in\Z^4} \mathbf{P}_v(
		\tau_i<\tau_{i}^c,\,\LE(X^{\tau_i})^\leftarrow\in\mathscr{A}\vert^i_{v,\eta},\, \LE(X^{\tau_i}) \text{ $\delta$-bad}).
	\end{multline}
	If $\LE(X^{\tau_i})$ is $\delta$-good then we have by \eqref{eq:Ttilde_bounds_T} and \eqref{eq:delta_good_Ttilde_upper} that $T_i:=T_{\eta_i,\eta\setminus\{\eta_i\}}(|\LE(X^{\tau_i})|) \leq C_1 n (\log n)^{1/3+\delta}$ for some universal constant $C_1$, and hence that 
	\begin{align}
		&\mathbf{P}_v(\tau_i<\tau_{i}^c,\, \LE(X^{\tau_i})^\leftarrow\in\mathscr{A}\vert^i_{v,\eta},\, \LE(X^{\tau_i}) \text{ $\delta$-good})
		\nonumber\\
		&\hspace{2cm}\leq \mathbf{P}_v(\tau_i<\tau_{i}^c,\, \LE(X^{\tau_i})^\leftarrow\in\mathscr{A}\vert^i_{v,\eta}, \,T_i
		\leq C_1 n(\log n)^{1/3+\delta})
		\nonumber\\
		&\hspace{2cm}\leq \mathbf{P}_v\big(\tau_i<\tau_{i}^c,\, \LE(X^{\tau_i})^\leftarrow\in\mathscr{A}\vert^i_{v,\eta},\, \abs{T_{i}-\tau_{i}}\geq\lambda n\big)
		\nonumber\\
		&\hspace{5cm}+\mathbf{P}_v\big(\tau_i<\tau_{i}^c,\, \LE(X^{\tau_i})^\leftarrow\in\mathscr{A}\vert^i_{v,\eta},\, \tau_i\leq C_1 n(\log n)^{1/3+\delta}+\lambda n\big)
		\label{eq:union_bound_step_2}
	\end{align}
	for every $\lambda >0$.
	The first term on the right hand side of \eqref{eq:union_bound_step_2} is bounded above by $C_2\lambda^{-1} \mathbf{P}_v(\tau_i<\tau_{i}^c,\, \LE(X^{\tau_i})^\leftarrow\in\mathscr{A}\vert^i_{v,\eta})$ for some universal constant $C_2$ by Lemma \ref{lemma:conc_T}, so that taking $\lambda=2C_2$, substituting \eqref{eq:union_bound_step_2} into \eqref{eq:bad_split}  and rearranging 
	yields that
	\begin{multline} \label{eq:combined_2}
		\sum_{v\in\Z^4} \mathbf{P}_v\bigl(\tau_i<\tau_{i}^c,\,\LE(X^{\tau_i})^\leftarrow\in\mathscr{A}\vert^i_{v,\eta}\bigr) 
		\leq 2 \sum_{v\in\Z^4} \mathbf{P}_v(\tau_i<\tau_{i}^c,\,\LE(X^{\tau_i})^\leftarrow\in\mathscr{A}\vert^i_{v,\eta},\, \tau_i \leq C_3n (\log n)^{1/3+\delta}) \\+2\sum_{v\in\Z^4} \mathbf{P}_v(\tau_i<\tau_{i}^c,\,\LE(X^{\tau_i})^\leftarrow\in\mathscr{A}\vert^i_{v,\eta},\, \LE(X^{\tau_i}) \text{ $\delta$-bad}),
	\end{multline}
	where $C_3=C_3(\delta)$ has been chosen so that $C_1 n (\log n)^{1/3+\delta} +2C_2 n\leq C_3 n(\log n)^{1/3+\delta}$ for every $n\geq 2$.

	\medskip

	We next bound the second term on the right hand side of \eqref{eq:combined_2}.
	Since the typical time of a length $n$ path is always $O(n \log n)$, it follows by the same argument used to derive \eqref{eq:combined_2} from \eqref{eq:union_bound_step_2} that there exists a constant $C_4$ such that
	\begin{multline*}
		\sum_{v\in\Z^4} \mathbf{P}_v(\tau_i<\tau_{i}^c,\,\LE(X^{\tau_i})^\leftarrow\in\mathscr{A}\vert^i_{v,\eta},\, \LE(X^{\tau_i}) \text{ $\delta$-bad}) \\
		\leq 2\sum_{v\in\Z^4} \mathbf{P}_v(\tau_i<\tau_{i}^c,\,\LE(X^{\tau_i})^\leftarrow\in\mathscr{A}\vert^i_{v,\eta},\, \tau_i \leq C_4 n \log n).
	\end{multline*}
Thus, taking a union bound over the possible values of $\tau_i$, we have that
	\begin{align*}
		\sum_{v\in\Z^4} \mathbf{P}_v(\tau_i<\tau_{i}^c,\,\LE(X^{\tau_i})^\leftarrow\in\mathscr{A}\vert^i_{v,\eta},\, \LE(X^{\tau_i}) \text{ $\delta$-bad})
		&\leq 
		2 \sum_{v\in\Z^4}  \sum_{k=0}^{\ceil{C_4n\log n}}
		\mathbf{P}_v(X_k=\eta_i, \LE(X^k) \text{ $\delta$-bad})
    \end{align*}
and hence by symmetry that
    \begin{multline}
\sum_{v\in\Z^4} \mathbf{P}_v(\tau_i<\tau_{i}^c,\,\LE(X^{\tau_i})^\leftarrow\in\mathscr{A}\vert^i_{v,\eta},\, \LE(X^{\tau_i}) \text{ $\delta$-bad})
    \\\leq 2 \sum_{v\in\Z^4}  \sum_{k=0}^{\ceil{C_4n\log n}}
    \mathbf{P}_{\eta_i}(X_k=v, \LE(X^k) \text{ $\delta$-bad})
    =2  \sum_{k=0}^{\ceil{C_4n\log n}}
    \mathbf{P}_{\eta_i}( \LE(X^k) \text{ $\delta$-bad}) \preceq_{\delta,p} n(\log n)^{1-p}\label{eq:bad_ineq}
	\end{multline}
	for every $n\geq 2$. 
	
	\medskip
	
	Next, we consider the first term on the right hand side of \eqref{eq:combined_2}. We write $B=\{\tau_i<\tau_{i}^c,\,\LE(X^{\tau_i})^\leftarrow\in\mathscr{A}\vert^i_{v,\eta},\, \tau_i \leq C_3n (\log n)^{1/3+\delta}\}$ and wish to estimate $\sum_{v\in\Z^4} \bP_v(B)$.
	To do this, we split the event $B$ according to how far the walk travels before hitting $\eta_i$, yielding the union bound
	\begin{multline}
		\mathbf{P}_v(B) \leq \mathbf{P}_v\big(B, \sup_{0\leq m \leq\tau_i} \norm{X_m-\eta_i}\geq n^{1/2} (\log n)^{1/6+\delta}\big)\\
		+\mathbf{P}_v\big(B, \sup_{0\leq m \leq\tau_i} \norm{X_m-\eta_i}< n^{1/2} (\log n)^{1/6+\delta}\big).
		\label{eq:big_small_displacement_union_bound}
	\end{multline}
	For the first of these terms, we bound
	\begin{multline*}
		\mathbf{P}_v\big(B, \sup_{m\leq\tau_i} \norm{X_m-\eta_i}\geq n^{1/2} (\log n)^{1/6+\delta}\big)\\
		\leq  \sum_{k=0}^{\lceil C_3n(\log n)^{1/3+\delta}\rceil}\mathbf{P}_v(X_k=\eta_i,\sup_{m\leq k}\norm{X_m-\eta_i} \geq n^{1/2} (\log n)^{1/6+\delta}).
	\end{multline*}
	Summing over $v$ and using time-reversal gives that
	\begin{align}
		&\sum_{v\in \Z^4} \mathbf{P}_v\big(B, \sup_{m\leq\tau_i} \norm{X_m-\eta_i}\geq n^{1/2} (\log n)^{1/6+\delta}\big)\nonumber\\ 
		&\hspace{4cm}\leq \sum_{k=0}^{\lceil C_3n(\log n)^{1/3+\delta}\rceil}\mathbf{P}_{\eta_i}\Biggl(\sup_{m\leq k}\norm{X_m-\eta_i} \geq n^{1/2} (\log n)^{1/6+\delta}\Biggr) \nonumber\\
		&\hspace{4cm}\preceq n (\log n)^{1/3+\delta} \mathbf{P}_0 \Biggl(\sup_{m\leq \lceil C_3n(\log n)^{1/3+\delta}\rceil}\norm{X_m} \geq n^{1/2} (\log n)^{1/6+\delta}\Biggr)
		\nonumber\\
		&\hspace{4cm}\preceq n (\log n)^{1/3+\delta} e^{-c_1 (\log n)^{\delta}} \preceq_{\delta,p} n (\log n)^{-p}
		\label{eq:big_displacement}
	\end{align}
	for some constant $c_1>0$, 
	where the first inequality in the last line follows by e.g.\ the maximal version of Azuma-Hoeffding \cite[Section 2]{McDiarmid1998}.
	
	\medskip
	
	Substituting the estimates  \eqref{eq:bad_ineq} and \eqref{eq:big_displacement}  into \eqref{eq:combined_2} in light of \eqref{eq:big_small_displacement_union_bound} yields that there exists a constant $C_{\delta,p}$ such that
	\begin{multline} \label{eq:combined_3}
		\sum_{v\in\Z^4} \mathbf{P}_v(\tau_i<\tau_{i}^c,\,\LE(X^{\tau_i})^\leftarrow\in\mathscr{A}\vert^i_{v,\eta})\\ 
		\preceq \sum_{v\in\Z^4}  \mathbf{P}_v\big(\tau_i<\tau_{i}^c,\,\LE(X^{\tau_i})^\leftarrow\in\mathscr{A}\vert^i_{v,\eta},\, \sup_{m\leq\tau_i} \norm{X_m-\eta_i}\leq n^{1/2} (\log n)^{1/6+\delta}\big) +C_{\delta,p} n(\log n)^{-p}.
	\end{multline}
	Now $\LE(X^{\tau_i})\subseteq (X_m)_{m\leq\tau_i}$, and so applying Wilson's algorithm, we have
	\begin{multline*}
		\mathbf{P}_v\big(\tau_i<\tau_{i}^c,\,\LE(X^{\tau_i})^\leftarrow\in\mathscr{A}\vert^i_{v,\eta} \sup_{m\leq\tau_i} \norm{X_m-\eta_i}\leq n^{1/2} (\log n)^{1/6+\delta}\big)\preceq\\  
		\P^\eta(0\wedge v=\eta_i, \Gamma(0\wedge v,v)\in\mathscr{A}\vert^i_{v,\eta},\,\text{ and }\Gamma(v,0 \wedge v)\subseteq \Lambda(0\wedge v,n^{1/2}(\log n)^{1/6+\delta})).
	\end{multline*}
	Substituting this inequality into \eqref{eq:combined_3} and summing over $i\in\mathcal{I}$ yields
	\begin{multline*}
		\sum_{i\in\mathcal{I}}\sum_{v\in\Z^4} \mathbf{P}_v(\tau_i<\tau_{i}^c,\,\LE(X^{\tau_i})^\leftarrow\in\mathscr{A}\vert^i_{v,\eta})
		\preceq\\
		\sum_{v\in\Z^4} \P^\eta(\Gamma(0,v)\in\mathscr{A},\, \Gamma(0,0 \wedge v)\subseteq \Lambda(r)\text{ and }\Gamma(v,0 \wedge v)\subseteq \Lambda(0\wedge v,n^{1/2}(\log n)^{1/6+\delta}))+C_{\delta,p} n^2(\log n)^{-p},
	\end{multline*}
	since $\abs{\mathcal{I}}\leq n+1$. Substituting this inequality into \eqref{eq:B(0,n)_SRW_expansion} and taking expectations over $\eta$ yields the claimed inequality \eqref{eq:displacement_split_m}.
\end{proof}

\textbf{Containment of balls.} We now turn our attention to the proof of Proposition \ref{prop:extr_containment}. We begin by showing that it is very unlikely for $\fT$ to include a crossing of an annulus that it shorter than it should be by a large (i.e.\ non-sharp) polylogarithmic factor. We write $\partial \Lambda(r)$ for the set of vertices in $\Z^4$ with $\|x\|_\infty=r$.

\begin{lemma}
\label{lem:annulus_crossing}
Let $\fT$ be the uniform spanning tree of $\Z^4$ and for each $r,n\geq 1$ let $\mathscr{E}(r,n)$ be the event that there exists a path in $\fT$ from $\partial \Lambda(r)$ to $\partial \Lambda(4r)$ that has length at most $n$. Then
\[
\P\left(\mathscr{E}\left(r,\lceil r^2(\log r)^{-3}\rceil \right)\right) = \exp\left[-\Omega((\log r)^2)\right]
\]
as $r\to \infty$.
\end{lemma}

\begin{proof}[Proof of Lemma~\ref{lem:annulus_crossing}]
Fix $r\geq 2$, let $n=\lceil r^2(\log r)^{-3}\rceil $, and write $\mathscr{E}=\mathscr{E}(r,n)$. If $\mathscr{E}$ holds, there must exist a pair of points $x\in \partial \Lambda(r)$ and $y\in \partial \Lambda(4r)$ such that the path connecting $x$ and $y$ in $\fT$ is contained in the box $\Lambda(4r)$ and has length at most $n$. Considering separately the case that $x\wedge y$ belongs to $\Lambda(2r)$ or not yields the union bound
\begin{align*}
\P(\mathscr{E})&\leq \sum_{y\in \partial \Lambda(4r)} \sum_{z\in \Lambda(2r)} \P\left(z \in \Gamma(y,\infty), |\Gamma(y,z)|\leq n\right) + \sum_{x\in \partial \Lambda(r)} \sum_{z\in \Lambda(4r)\setminus \Lambda(2r)}\P\left(z \in \Gamma(x,\infty), |\Gamma(x,z)|\leq n\right),
\end{align*}
and using Wilson's algorithm to convert this into a loop-erased random walk quantity yields that
\begin{align}
\P(\mathscr{E})
&\leq\sum_{y\in \partial \Lambda(4r)} \sum_{z\in \Lambda(2r)} \sum_{k=0}^n \bP_y\left(\LE(X)_k=z\right) + \sum_{x\in \partial \Lambda(r)} \sum_{z\in \Lambda(4r)\setminus \Lambda(2r)}\sum_{k=0}^n\bP_x\left(\LE(X)_k=z\right)\nonumber\\
&= \sum_{y\in \partial \Lambda(4r)} \sum_{k=0}^n \bP_y\left(\LE(X)_k\in \Lambda(2r)\right) + \sum_{x\in \partial \Lambda(r)} \sum_{k=0}^n\bP_x\left(\LE(X)_k\in \Lambda(4r)\setminus \Lambda(2r) \right)\nonumber\\
&\preceq r^3 n \bP_0\left(\max_{0\leq k\leq n}\|\LE(X)_k\|_\infty\geq r \right).
\label{eq:annulus_crossing_union}
\end{align}
We will bound this probability using the weak $L^1$ method as introduced in \cite[Section 6.2]{hutchcroft2020logarithmic}, which can be thought of as a simple special case of the typical time theory.
Conditional on the loop-erased random walk $\LE(X)$, we have as in \cite[Lemma 5.3]{MR4055195} that the sequence of random variables $(\ell_{i+1}(X)-\ell_i(X))_{i\geq 0}$ are conditionally independent and satisfy
\[
\bP_0(\ell_{i+1}(X)-\ell_i(X) = m \mid \LE(X)) \leq p_{m-1}(0,0) \preceq \frac{1}{m^2}
\]
for every $m\geq 1$, and it follows from Vershynin's weak triangle inequality for the weak $L^1$ norm \cite{VershyninweakL1} as explained in \cite[Section 6.2]{hutchcroft2020logarithmic} that
\[
\bP_0(\ell_{n}(X) \geq m \mid \LE(X)) \preceq \frac{n \log n}{m}
\]
for every $n\geq 2$ and $m\geq 1$. As such, there exists a constant $C$ such that
\begin{align*}
\bP_0\left(\max_{0\leq k\leq n}\|\LE(X)_k\|_\infty\geq r \right)
&\leq 2 \bP_0\left(\max_{0\leq k\leq n}\|\LE(X)_k\|_\infty\geq r,\, \ell_n(X) \leq Cn\log n \right)\\
&\leq 2 \bP_0\left(\max_{0\leq i\leq Cn\log n}\|X_i\|_\infty\geq r \right)
\\ &\preceq \exp\left[-\Omega\left(\frac{r^2}{n \log n}\right)\right]
\preceq \exp\left[-\Omega\left((\log r)^2\right)\right].
\end{align*}
where we have used the maximal version of Azuma-Hoeffding in the last line \cite[Section 2]{McDiarmid1998}. The claim follows by substituting this estimate into \eqref{eq:annulus_crossing_union} and using that $r^3n =r^{O(1)}=\exp[o((\log r)^2)]$.
\end{proof}

Before proceeding with the deduction of Proposition~\ref{prop:extr_containment} from Proposition~\ref{prop:subsume} and Lemma~\ref{lem:annulus_crossing}, we will first introduce some more tools from \cite{MR4055195,hutchcroft2020logarithmic}.
We begin by defining a variant of the uniform spanning tree known as the \textit{0-wired uniform spanning forest}, which was first introduced by J\'{a}rai and Redig \cite{JarRed08} as part of their work on the Abelian sandpile model. Let $(V_n)_{n\geq 0}$ be an exhaustion of $\Z^4$ by finite connected sets. For each $n\geq 0$, let $G_n^{*}$ be the graph obtained by identifying (a.k.a.\ wiring) $\Z^4\setminus V_n$ into a single point denoted by $\partial_n$. Let $G_n^{*0}$ be the graph obtained by identifying $0$ with $\partial_n$ in in $G_n^*$. The \textbf{0-wired uniform spanning forest} is then the weak limit of the uniform spanning trees on $G_n^{*0}$ as $n\rightarrow\infty$, which is well-defined and does not depend on the choice of exhaustion \cite[\S3]{LMS08}. Lyons, Morris and Schramm \cite{LMS08} proved that the component of the origin in the $0$-wired forest is finite almost surely,
and, since the entire 0-wired forest is stochastically dominated by the uniform spanning tree by \cite[Theorem 4.6]{LP:book}, and the definitions ensure that every component other than that of the origin is infinite, the rest of the vertices of $\Z^4$ are contained in a single infinite one-ended component almost surely.

\medskip

\textbf{The stochastic domination property.} We let $\mathfrak{T}_0$ be the component of $0$ in the $0$-wired UST. Lyons, Morris and Schramm \cite[Proposition 3.1]{LMS08} proved that $\mathfrak{T}_0$ stochastically dominates $\mathfrak{P}(0)$, which we recall denotes the past of the origin in $\fT$. In \cite{MR4055195}, a stronger version of this stochastic domination property was derived, the relevant parts of which we restate below in our context. Given that the UST of $\Z^4$ is connected and one-ended, we can, in a unique manner, add an orientation to each edge in $\mathfrak{T}$ so that each vertex in the tree has exactly one oriented edge emanating from it. By abuse of notation, we denote the resulting oriented tree by $\fT$ as we do in the unoriented case. The oriented $0$-wired spanning forest $\mathfrak{F}_0$ is generated similarly, but with the edges in the finite component all oriented towards the origin.
Lastly, we generalise the notion of the \textit{past}: given an arbitrary oriented forest $F$, we define the past of a vertex $v\in F$, denoted $\mathrm{past}_F(v)$, to be the set of vertices $u$ with a directed path $\gamma$ in $F$ emanating from $u$ and ending at $v$.
\begin{lemma}[Stochastic Domination] \label{lemma:stoch_dom}
	Let $\fT$ be the oriented uniform spanning tree of $\Z^4$, and let $\mathfrak{F}_0$ be the oriented $0$-wired  uniform spanning forest of $\Z^4$. Let $K$ be a finite set of vertices in $\Z^4$ and let $\Gamma(K)=\cup_{u\in K} \Gamma(u,\infty)$. Then for every increasing event $\mathscr{A}\subseteq\{0,1\}^{E(\Z^4)}$ we have that \[
	\pr\left(\mathrm{past}_{\mathfrak{F}\setminus \Gamma(K)}(0)\in\mathscr{A}\mid \Gamma(K)\right)\leq \pr\big(\mathfrak{T}_0\in\mathscr{A}\big).
	\]
\end{lemma}

 We will also utilize the following result of \cite{hutchcroft2020logarithmic}. For any subset $A$ of $\Z^4$ containing the origin, let $\Rext(A)$ be the maximal $\ell^1$ distance between the origin and a vertex of $A$.

\begin{theorem}[\!\!\cite{hutchcroft2020logarithmic}, Theorem 1.6]\label{theorem:ext_bounds} Let $\mathfrak{T}_0$ be the component of the origin in the $0$-wired uniform spanning tree of $\Z^4$. Then
	\[
	\pr\Big(\Rext(\mathfrak{T}_0)\geq n\Big)\asymp \frac{(\log n)^{1+o(1)}}{n^2}
	\]
 for every $n\geq 2$.
\end{theorem}

\begin{remark}
For the proof of Proposition~\ref{prop:extr_containment} it would suffice to have the weaker bound in which $(\log n)^{1+o(1)}$ is replaced by $(\log n)^{O(1)}$, which is significantly easier to prove. (That is, it can be proven by the high-dimensional methods of \cite{MR4055195} without needing a careful analysis of the four-dimensional case.)
\end{remark}

 With these tools in hand we proceed to the proof of Proposition \ref{prop:extr_containment}.
\begin{proof}[Proof of Proposition \ref{prop:extr_containment}]
	Fix $\delta\in(0,1]$, and fix an integer $n\geq2$.
Let $\eta$ be the future of the origin in the uniform spanning tree $\mathfrak{T}$ and let $r=\lceil n^{1/2} (\log n)^{1/6+\delta}\rceil$.
We write 
\[\{\fB(n) \nsubseteq \Lambda(8r)\} \subseteq \mathscr{F}\cup \mathscr{E} \cup \mathscr{A},\]
where
$\mathscr{F}=\{\eta[0,n]\nsubseteq \Lambda(r)\}$ is the event that the first $n$ steps of the future are not contained in the box of radius $r$, $\mathscr{E}=\mathscr{E}(r,\lceil r^2/(\log r)^{-3}\rceil )$ is the event defined in Lemma~\ref{lem:annulus_crossing}, and $\mathscr{A}$ is the event $\{\fB(n) \nsubseteq \Lambda(8r)\}\setminus (\mathscr{F}\cup\mathscr{E})$.
We have already shown in Lemma~\ref{lem:annulus_crossing} that the probability of $\mathscr{E}$ is much smaller than required for $n$ sufficiently large. For the event $\mathscr{F}$, we use Wilson's algorithm to compute that
  \begin{align*}
  \pr\big(\mathscr{F}\big)&=\pb_0\big(\LE(X)^n\nsubseteq\Lambda(r)\big)\\
  &\leq \pb_0\big(\ell_n>2n(\log n)^{1/3})+\pb_0\left(\max_{0\leq k \leq 2n(\log n)^{1/3}} \|X_k\|_\infty > r\right)\\
  &\preceq \frac{\log\log n}{(\log n)^{2/3}} + \exp\left[-\Omega((\log n)^{\delta})\right]\preceq_\delta \frac{\log\log n}{(\log n)^{2/3}}
  \end{align*}
  as required,
 where the second inequality follows by Theorem \ref{theorem:LERWconc} for the bound on $\ell_n$, and e.g.\ the maximal version of Azuma-Hoeffding \cite[Section 2]{McDiarmid1998} for the bound on the displacement of the simple random walk.

We now bound the probability of $\mathscr{A}$. Observe that if $\mathscr{A}$ holds then there exists an integer $0\leq i\leq n-1$ such that $\mathfrak{P}(\eta_i,n)$ is not contained in $\Lambda(8r)$. Since $\mathscr{E}$ does not hold, we must also have that every crossing of the annulus $\Lambda(4r)\setminus \Lambda(r)$ has length at least $r^2/(\log r)^3$, and it follows that there must exist a collection of at least $r^2/(\log r)^3$ points $y\in (\fB(n)\setminus\eta[0,n]) \cap (\Lambda(4r)\setminus \Lambda(r))$ such that $\fP(y,n)$ has extrinsic diameter at least $4r$. Summing over all possible such points, applying Markov's inequality yields, and  using the stochastic domination lemma (Lemma~\ref{lemma:stoch_dom}) yields that
\begin{align*}
	\P(\mathscr{A})&\leq \frac{(\log r)^3}{r^2}\sum_{y\in\Lambda(4r)\setminus \Lambda(r)}\pr(y\in\mathfrak{B}(n)\setminus\eta[0,n] \text{ and } \mathrm{diam}(\mathfrak{P}(y))\geq 4r)\\
	&\leq \frac{(\log r)^3}{r^2}\sum_{y\in\Lambda(4r)\setminus \Lambda(r)} \pr\big(y\in\mathfrak{B}(n))  \pr\big(\Rext(\mathfrak{T}_0) \geq2r\big)
\\&=\frac{(\log r)^3}{r^2}\mathbb{E}\ \big\vert\{y\in\mathfrak{B}(n):y\notin \Lambda(r)\}\big\vert \pr\big(\Rext(\mathfrak{T}_0) \geq 2r\big),
  \end{align*}
  and it follows from Proposition~\ref{prop:subsume} and Theorem~\ref{theorem:ext_bounds} that
  \begin{align*}
	\P(\mathscr{A}) \preceq_{\delta,p} \frac{(\log r)^3}{r^2} \frac{n^2}{(\log n)^p}\cdot \frac{(\log r)^{1+o(1)}}{r^2} \preceq (\log n)^{10/3-4\delta-p+o(1)},
\end{align*}
for every $p\geq 1$. Taking $p=10$, say, yields a bound that is stronger than required and completes the proof.
\end{proof}

\subsection{Lower bounds}
\label{subsec:volume_lower}

In this section we prove the following proposition, which implies the lower bounds of Theorem~\ref{thm:main_geometric_theorem}. Note that, in contrast to Proposition~\ref{prop:volume_upper}, we do not lose any $(\log n)^{\pm o(1)}$ factors in this bound.

			\begin{proposition} \label{prop:volume_lower}
			Let $\fT$ be the uniform spanning tree of $\Z^4$. Then 
				\[
				 \abs{\mathfrak{B}(n)} = \bOmega\left(\frac{n^2}{(\log n)^{1/3}}\right)
				\]
				as $n\to\infty$.
			\end{proposition}

\begin{remark}
The proof yields the explicit lower tail bound
\[
\P\left(\abs{\mathfrak{B}(n)} \leq \frac{n^2}{\lambda (\log n)^{1/3}}\right) \preceq \lambda^{-1/5}
\]
for every $n\geq 3$ and $1\leq \lambda \leq \log n$. Presumably this bound is far from optimal.
      \end{remark}

We will prove this proposition by estimating the mean and variance of certain random variables that lower bound $\abs{\mathfrak{B}(n)}$. We expect $\abs{\mathfrak{B}(n)}$ to be unconcentrated\footnote{Indeed, it should converge under appropriate rescaling to the volume of a ball in the ICRT (Process 2 in \cite{aldous1991continuum}), which is not deterministic.
}, so its variance should be of the same order as its second moment and applying Chebyshev directly to $|\fB(n)|$ should not be a viable method to prove lower tail bounds. Instead we calculate the mean and variance of a certain `good' portion of the uniform spanning tree within a certain radius of the spine. We choose this radius according to how deep into the lower tail of the volume we wish to control: the lower we take this radius, the deeper into the tail we bound. The precise meaning of `good' we will use is engineered precisely to make the later parts of the proof go through cleanly.

\medskip

Our first task is to set up the relevant definitions. Recall that $\bP_z$ denotes the law of a simple random walk $X$ on $\Z^4$ started at $z$ for each $z\in \Z^4$.
\cite[Theorem 7.4]{hutchcroft2020logarithmic} states that if $\pb_{0,\Lambda(r)}$ denotes  the joint law of two independent random walks $X$ and $Y$ started at $0$ and at a uniform point of $\Lambda(r)$ respectively, then
\begin{equation}
\label{eq:intersect_inside_box}
\pb_{0,\Lambda(r)}(X \cap Y \cap \Lambda(r) \neq \emptyset) \asymp \frac{1}{\log r}
\end{equation}
for $r\geq 2$.
Fix $\alpha>0$ and $r\geq 2$.
We say a path $\gamma$ in $\Z^4$ is \textbf{$(\alpha,r)$-good} if
\[
\sum_{z\in \Lambda(\gamma_0,6r)}\mathbf{P}_z\big(\text{hit }\gamma\cap\Lambda(\gamma_0,6r)\big)\leq \alpha\frac{r^4}{\log r},
\]
and say that $\gamma$ is \textbf{$(\alpha,r)$-bad} otherwise.
We note that 
\begin{align}\label{eq:bad_low_prob}
\pb_0(X\ \text{is $(\alpha,r)$-bad})&=
\bP_{0,\Lambda(6r)} \left(|\Lambda(6r)| \bP_{0,\Lambda(6r)}\left(X \cap Y \cap \Lambda(6r) \neq \emptyset \mid X \right) > \alpha \frac{r^4}{\log r}\right)
\preceq \alpha^{-1}
\end{align}
by \eqref{eq:intersect_inside_box} and Markov's inequality.
 Crucially, we also observe that being $(\alpha,r)$-bad is an increasing property of a path in the sense that if $\gamma$ and $\tilde \gamma$ are two paths satisfying $\gamma_0=\tilde \gamma_0$  and $\gamma\subseteq \tilde \gamma$, then $\tilde \gamma$ is $(\alpha,r)$-bad whenever $\gamma$ is $(\alpha,r)$-bad. We will apply this to bound the probability that a loop-erased random walk is bad in terms of the probability that the corresponding \emph{simple} random walk is bad.

\medskip

Condition on the future of the origin $\eta:=\Gamma(0,\infty)$ in the uniform spanning tree $\fT$ and for each $x\in\Z^4$ and $r\geq 3$ consider the random set
\[
M_{\alpha}(x,r) = \left\{y\in \Lambda(x,3r):\Gamma(y,0\wedge y)\subseteq \Lambda(x,3r),\, \abs{\Gamma(y,0\wedge y)}\leq \frac{r^2}{(\log r)^{1/3}},\text{ and } \Gamma(y,0\wedge y) \text{ is $(\alpha,r)$-good}\right\}.
\]
The key step in the proof of Proposition~\ref{prop:volume_lower} is to bound the conditional mean and variance of $|M_\alpha(x,r)|$ in terms of the \emph{capacity} of $\eta$. Here we recall that the \textbf{capacity} (a.k.a.\ conductance to infinity) of a set $A \subseteq \Z^4$ is defined to be
\[
\mathrm{Cap}(A)=\sum_{a\in A} \deg(a)\bP_a(\text{never return to $A$ after time zero}) = 8\sum_{a\in A} \bP_a(\text{never return to $A$ after time zero}).
\]
The two relevant estimates are as follows, where we write $\mathrm{Var}^\eta$ for the conditional variance given $\eta$:

\begin{proposition}\label{prop:lower_bound_expectation} 
There exist $\alpha_0>0$ and $r_0>0$ such that if $\alpha \geq \alpha_0$ then
	\[
	\mathbb{E}^\eta\abs{M_\alpha(x,r)}\succeq r^{2}\mathrm{Cap}(\eta\cap \Lambda(x,r))
	\]
for every $x\in \Z^4$ and every $r\geq r_0$.
\end{proposition}
\begin{proposition}\label{prop:upper_bound_variance}
For each $\alpha>0$ we have
	\[
	\mathrm{Var}^\eta(\abs{M_\alpha(x,r)})\preceq\alpha \frac{r^6}{\log r}\mathrm{Cap}(\eta\cap \Lambda(x,3r)).
	\]
for every $x\in \Z^4$ and every $r\geq 2$.
\end{proposition}

We will require the following variational formula for the capacity proved in \cite[Lemma 2.3]{JAIN1973795}.
Recall that the Green's function on $\Z^4$ is defined by
\[G(x,y) = \frac{1}{\deg y}\mathbf{E}_x\sum_{n\geq 0}\mathbbm{1}(X_n=y)=\frac{1}{8}\mathbf{E}_x\sum_{n\geq 0}\mathbbm{1}(X_n=y),\]
where $X$ is a simple random walk on $\Z^4$ and $x,y\in\Z^d$.
\begin{lemma}\label{lem:capacity_formulation}
	The capacity of a set $S\subset\Z^4$ can be expressed as 	
	\begin{equation}\label{eq:capacity_formulation}
	\mathrm{Cap}(S)^{-1} = \inf\left\{\sum_{u,v\in S} G(u,v)\mu(u)\mu(v):\mu \text{ is a probability measure on $S$}\right\}.
	\end{equation}
\end{lemma}


\begin{proof}[Proof of Propostion \ref{prop:lower_bound_expectation}]
Fix $x\in \Z^4$, $r\geq 1$ and $\alpha>0$.
	We assume that $\mathrm{Cap}(\eta\cap\Lambda(x,r))>0$ or else the proposition is trivial.
	We let $n=\lfloor r^2(\log r)^{-1/3}\rfloor$ and $N=\floor{\lambda r^2}$ where $\lambda\in(0,1/2)$ is a parameter that will later be taken to be a small constant.
 Let $V$ be a uniform random element of $\Lambda(x,3r)$, let $X=(X_m)_{m\geq 0}$ be a random walk started at $V$, and let $\bP$ denote the joint law of $V$ and $X$.
	Let $\sigma$ be the time at which $X^{N}$ hits $\eta\cap\Lambda(x,3r)$ and let $\tau$ be the time $X^{N}$ first exits $\Lambda(x,3r)$. Each of these stopping times is defined to be infinite if the relevant event does not occur before or at time $N$.
	We let $\mu$ be a measure which minimises the right hand side of \eqref{eq:capacity_formulation} when $S=\eta\cap\Lambda(x,r)$ and define the random variable
	\[A_r=\mathbbm{1}(\sigma<\tau,\abs{\LE(X^\sigma)}\leq n,\LE(X^\sigma) \text{ good})\sum_{w\in \eta\cap\Lambda(x,r)} \sum_{j=0}^{N}  \mu(w)\mathbbm{1}(X_j=w), \]
	where to save on notation we have and will abbreviate $(\alpha,r)$-good and $(\alpha,r)$-bad to  good and bad respectively. The weight $\mu$ is included in the definition of $A_r$ since it makes the second moment of $A_r$ easier to control; this is closely related to the theory of \emph{Martin capacity} as developed in \cite{MR1349175}.
	 An application of Wilson's algorithm implies that
	\[
	\mathbb{E}^\eta\abs{M_\alpha(x,r)}\geq \sum_{v\in\Lambda(x,3r)}\pb(A_r>0 \mid V=v)=\abs{\Lambda(x,3r)}\pb(A_r>0),
	\]
	so that to prove the proposition we need only demonstrate that there exists $\alpha_0,r_0>0$ such that
	\[
	\pb (A_r>0)\succeq r^{-2}\mathrm{Cap}(\eta\cap\Lambda(x,r))
	\]
	for every $\alpha \geq \alpha_0,r\geq r_0$, 
	where we emphasize that the constant implied by the $\succeq$ on the right hand side is independent of $\eta$ and $r$.
	We do so by  proving that
	\begin{multicols}{2}
		\noindent\begin{equation}\label{eq:lower_bound_exp_A}
	\mathbf{E}{A_r}\succeq r^{-2}
		\end{equation}
		\begin{equation} \label{eq:upper_bound_exp_A}
	\mathbf{E}{A_r^2}\preceq r^{-2}\mathrm{Cap}^{-1}(\eta\cap\Lambda(x,r))
		\end{equation}
	\end{multicols}
 \noindent for appropriately large $\alpha,r$ and an appropriately small constant value of $\lambda$; 
	once \eqref{eq:lower_bound_exp_A} and \eqref{eq:upper_bound_exp_A} are established the claim will follow since, by Cauchy-Schwartz,
	\[
	\mathbb{E}^\eta\abs{M_\alpha(x,r)}\succeq r^4\pb(A_r>0)\succeq r^4 \frac{\mathbf{E}[A_r]^2}{\mathbf{E}[A_r^2]}\succeq r^2 \mathrm{Cap}(\eta\cap\Lambda(x,r))
	\]
	as claimed.

\medskip

	 We begin by lower bounding the expectation of $A_r$.
	We decompose $A_r$ as 
		$A_r=E_r-D_r-C_r-B_r$,
	where
 	\begin{align*}
		B_r&=\mathbbm{1}(\sigma\geq\tau)\sum_{w\in \eta\cap\Lambda(x,r)} \sum_{j=0}^{N}  \mu(w)\mathbbm{1}(X_j=w),\\
		C_r&=\mathbbm{1}(\sigma<\tau,\abs{\LE(X^\sigma)}> n)\sum_{w\in \eta\cap\Lambda(x,r)} \sum_{j=0}^{N}  \mu(w)\mathbbm{1}(X_j=w),\\
		D_r&=\mathbbm{1}(\sigma<\tau,\abs{\LE(X^\sigma)}\leq n,\LE(X^\sigma) \text{ bad})\sum_{w\in \eta\cap\Lambda(x,r)} \sum_{j=0}^{N}  \mu(w)\mathbbm{1}(X_j=w), \qquad \text{ and }\\
		E_r&=\sum_{w\in \eta\cap\Lambda(x,r)} \sum_{j=0}^{N}  \mu(w)\mathbbm{1}(X_j=w).
	\end{align*}
	The random variable $E_r$ is the $\mu$-mass of the intersections of the random walk with the relevant part of $\eta$, i.e.\ $\eta\cap\Lambda(x,r)$. From $E_r$, we have subtracted the error term $B_r$ pertaining to the possibility that the walk exits the ball $\Lambda(x,3r)$ before hitting the relevant part of $\eta$; the term $C_r$ pertaining to the possibility that that the walk hits the relevant part of $\eta$ before exiting this ball, but has too long a loop erasure; and finally the term $D_r$ pertaining to the possibility that the walk hits the relevant part of $\eta$ before exiting this ball and has a suitably short loop erasure, but the loop erasure is \textit{bad}, as defined above.

\medskip

\noindent	\textbf{Lower bounding the expectation of $E_r$}: 
	First, we lower bound the expectation of $E_r$. We have by time-reversal that
	\begin{equation}\label{eq:Eineq}
	\begin{aligned}
		\mathbf{E}{[E_r]}&\geq \frac{1}{\abs{\Lambda(x,3r)}} \sum_{w\in \eta\cap\Lambda(x,r)} \mu(w) \sum_{j=0}^{N}  \sum_{v\in\Lambda(x,3r)}\pb_v(X_j=w)\\
		&\succeq r^{-4} \sum_{w\in \eta\cap\Lambda(x,r)} \mu(w) \sum_{j=0}^{N}\pb_w(X_j\in\Lambda(x,3r))\\
		&\succeq r^{-4}  \sum_{j=0}^{N}\pb_0(X_j\in\Lambda(0,2r))\succeq r^{-4} \sum_{j=0}^{N} 1 -\frac{j}{4r^2}\geq r^{-4}N\left(1-\frac{N}{4r^2}\right)\succeq \lambda r^{-2}(1-\lambda /4)\succeq \lambda r^{-2},
	\end{aligned}	
\end{equation}
	where the third inequality follows since $\sum_{w\in\eta\cap\Lambda(x,r)}\mu(w)=1$ and $\Lambda(w,2r)\subset\Lambda(x,3r)$ for $w\in\Lambda(x,r)$, the fourth inequality follows by e.g.\ the central limit theorem for the simple random walk, and the penultimate inequality holds if $r>1/\lambda$ (which is just the condition we need to avoid rounding $N$ down to zero).

\medskip

	\noindent \textbf{Upper bounding the expectation of $B_r$}:
	Next, we upper bound the expectation of $B_r$, which pertains to the possibility that the walk exits the ball $\Lambda(x,3r)$ before hitting the relevant part of $\eta$.
	We have 
	\begin{align*}
	\bE[B_r] &= \frac{1}{\abs{\Lambda(x,3r)}}   \sum_{v\in\Lambda(x,3r)}\sum_{w\in \eta\cap\Lambda(x,r)} \sum_{j=0}^{N}  \mu(w)\pb_v(X_j=w,\sigma\geq\tau)\\
	&\leq\frac{1}{\abs{\Lambda(x,3r)}}   \sum_{v\in\Lambda(x,3r)}\sum_{w\in \eta\cap\Lambda(x,r)} \sum_{j=0}^{N}  \mu(w)\pb_v(X_j=w,\tau\leq j)\\
	&\preceq  r^{-4} \sum_{v\in\Lambda(x,3r)}\sum_{w\in \eta\cap\Lambda(x,r)} \sum_{j=0}^{N}  \mu(w)\pb_w(X_j=v,\tau\leq j)\\
	&\preceq  r^{-4}N \sum_{w\in \eta\cap\Lambda(x,r)}  \mu(w) \pb_w(\tau\leq N)\preceq \lambda r^{-2} \bP_0\left(\sup_{0\leq i\leq N} \norm{X_i}_\infty\geq2r\right),
	\end{align*}
where the second inequality follows by time reversal of $X$, and the final inequality holds because the distance between any $w\in\eta\cap\Lambda(x,r)$ and $\partial\Lambda(x,3r)$ is greater than or equal to $2r$. Since $\mathbb{E}_o[\sup_{j\leq i}\norm{X_j}^2]\preceq i$ for $i\geq 0$, it follows by Markov's inequality that
\begin{equation}\label{eq:Bineq}
\bE[B_r]\preceq \lambda r^{-2} \frac{N}{r^2} \preceq \lambda^2 r^{-2}.
\end{equation}

\noindent \textbf{Upper bounding the expectation of $D_r$.}	
	We now upper bound the expectation of $D_r$, which pertains to the possibility that the walk hits the relevant part of $\eta$ before exiting this ball and has a suitably short loop erasure, but the loop erasure is bad.
		Observe that
	\begin{align*}
			\mathbf{E}[D_r]&\leq\mathbf{E}\left[\mathbbm{1}(\LE(X^\sigma)\text{ bad})\sum_{w\in\eta\cap \Lambda(x,r)} \sum_{j=0}^{N}  \mu(w)\mathbbm{1}(X_j=w)\right]
		\leq\mathbf{E}\left[\sum_{w\in\eta\cap \Lambda(x,r)} \sum_{j=0}^{N}  \mu(w)\mathbbm{1}(X\text{ bad}, X_j=w)\right]\\
		&\hspace{-1cm}\preceq r^{-4}\sum_{w\in\eta\cap \Lambda(x,r)} \sum_{j=0}^{N}  \mu(w)\sum_{v\in\Lambda(x,3r)} \pb_v(X\text{ bad}, X_j=w)\leq r^{-4}\sum_{w\in\eta\cap \Lambda(x,r)} \sum_{j=0}^{N}  \mu(w)\sum_v \pb_0(X\text{ bad}, X_j=w-v)\\ 
		&\hspace{-1cm}= r^{-4}\sum_{j=0}^N \pb_0(X \text{ bad})
		\preceq Nr^{-4}\pb_0(X\text{ bad})\preceq \alpha^{-1}Nr^{-4}\preceq\alpha^{-1}\lambda r^{-2},
	\end{align*}
where the second inequality follows as $\LE(X^\sigma)\subseteq X$, the fourth inequality follows by translation-invariance, and the penultimate inequality follows by \eqref{eq:bad_low_prob}.
	Combining this inequality with \eqref{eq:Bineq} and \eqref{eq:Eineq}, we can see that there exist positive constants $\alpha_0$ and $\lambda_0$ such that if $\alpha \geq \alpha_0$, $\lambda=\lambda_0$, and $r\geq 1/\lambda_0$ then
	\[
	\mathbf{E}[{E_r-D_r-B_r}]\succeq r^{-2}.
	\]
Thus, to complete the proof of \eqref{eq:lower_bound_exp_A}, it is sufficient to show that $\mathbf{E}{[C_r]}=o(r^{-2})$.

\medskip
 
 \noindent \textbf{Upper bounding the expectation of $C_r$}: 
 To bound the final term $C_r$, which pertains to the possibility that that the walk hits the relevant part of $\eta$ before exiting this ball, but has too long a loop erasure. We will need some understanding of the \textit{cut times} of a simple random walk. Recall that a time $t\geq 0$ is said to be a \textbf{cut time}, or \textit{loop-free time} of the random walk $X$ if $X[0,t]$ and $X(t,\infty)$ are disjoint. We observe that if $0 \leq s \leq t$ are cut times of $X$ then the loop-erasure of $X$ is equal to the concatenation of the loop-erasures of the portions of $X$ before $s$, between $s$ and $t$, and after $t$; this property allows us to decorrelate different parts of the loop-erased random walk.
 We use the following estimate of Lawler which demonstrates that the random walk on $\Z^4$ has a reasonably good supply of cut times.
 \begin{lemma}[\!\!\cite{MR1117680}, Lemma~7.7.4]
 	\label{lem:Lawlercuttimes}
 	\!Let $X$ be simple random walk on $\Z^4$. Then 
 	\[
 	\pb(\text{there are no cut times between times $n$ and $m$}) \preceq \frac{\log \log m}{\log m}.
 	\]
 	for every $3\leq n \leq m$ such that $|n-m| \geq m/(\log m)^6$.
 \end{lemma} 

Observe that if $\abs{\LE(X^\sigma)}>n$ then we must have that $\sigma>n$ and that if $X$ has a cut time in $[\sigma-n/4,\sigma]$, then $\abs{\LE(X^j)}\geq 3n/4$ for every $j\geq\sigma$. Therefore,
	\begin{equation}\label{eq:C_decomp}
		C_r \leq\sum_{w\in \eta\cap\Lambda(x,r)} \sum_{j=0}^{N}  \mu(w)\mathbbm{1}\Big(X_j=w,\abs{\LE(X^\sigma)}> n,n<\sigma\leq N\wedge j\Big)\leq C_r^\prime+C_r^{\prime\prime},
	\end{equation}
where 
	\begin{align*}
		C_r^\prime &=\sum_{w\in \eta\cap\Lambda(x,r)} \sum_{j=n+1}^{N}  \mu(w)\mathbbm{1}\Big(X_j=w,\text{$X$ has no cut time in $[\sigma-n/4,\sigma]$},n<\sigma\leq N\wedge j\Big)\qquad \text{and}\\
		C_r^{\prime\prime}&= \sum_{w\in \eta\cap\Lambda(x,r)} \sum_{j=n+1}^{N}  \mu(w)\mathbbm{1}\Big(X_j=w,\abs{\LE(X^j)}> \frac{3}{4}n\Big).
	\end{align*}
	We show that the expectation conditioned on $\eta$ of both $C_r^\prime$ and $C_r^{\prime\prime}$ is $o(r^{-2})$; we begin with the latter. We have
	\begin{align*}
		\mathbf{E} C_r^{\prime\prime}
		&\leq\frac{1}{\abs{\Lambda(x,3r)}}\sum_{w\in \eta\cap\Lambda(x,r)}\mu(w)\sum_{j=n}^N\sum_{v\in \Lambda(x,3r)} \pb_{v}\Big(X_j=w,\abs{\LE(X^j)}>\frac{3}{4}n\Big)\\
    &= \frac{1}{\abs{\Lambda(x,3r)}}\sum_{w\in \eta\cap\Lambda(x,r)}\mu(w)\sum_{j=n}^N\sum_{v\in \Lambda(x,3r)} \pb_{0}\Big(X_j=w-v,\abs{\LE(X^j)}>\frac{3}{4}n\Big)
    \\
    &\leq \frac{1}{\abs{\Lambda(x,3r)}}\sum_{w\in \eta\cap\Lambda(x,r)}\mu(w)\sum_{j=n}^N \pb_{0}\Big(\abs{\LE(X^j)}>\frac{3}{4}n\Big) \preceq r^{-4}\sum_{j=n}^N \pb_0\Big(\abs{\LE(X^j)}>\frac{3}{4}n\Big),
	\end{align*}
	where we used translation invariance in the second line.
	Observe for each $n\leq i\leq N$ that if $X$ has a cut time in $[i-i/(\log i)^6,i]$, then  $\abs{\LE(X^i)}\leq \abs{\LE_\infty(X^i)} +i/(\log i)^6$. Therefore,
	\begin{align} \label{eq:Cprimeprime_bound}
		\nonumber	 r^4\mathbf{E} C_r^{\prime\prime}&\preceq \sum_{i=n}^{N}  \pb_0\big(\abs{\LE(X^i)}> \frac{3}{4}n\big) \\
		\nonumber    &\leq \sum_{i=n}^N  \pb_0(\abs{\LE_\infty(X^i)}> (3/4)n-i/(\log i)^6 )+\pb_0\big(\text{$X$ has no cut times in $[i-i/(\log i)^6,i]$}\big)\\
		\nonumber  	 &\preceq  \sum_{i=n}^N  \pb_0(\rho_i> (3/4)n-i/(\log i)^6)+c\frac{\log\log i}{\log i}\\
		&\preceq N\frac{\log\log  n}{\log  n}+ \sum_{i=n}^N \frac{\log\log i}{(\log i)^{2/3} }\preceq  N\frac{\log\log n}{(\log n)^{2/3}}=o(r^2)
	\end{align}
	as required, where the third inequality follows by Lemma \ref{lem:Lawlercuttimes} and the fourth inequality follows from Theorem~\ref{theorem:LERWconc} and the fact that $\lambda <1/2$.
	Next, we upper bound the conditional expectation of $C_r^\prime$. Recalling the definitions $N=\floor{\lambda r^2}$ for some $\lambda\in(0,1/2)$ and $n=\floor{r^2(\log r)^{-1/3}}$, we can calculate that $N\leq n(\log n)^{1/3}$ for all $r\geq 2$. Define the sequence of times $T_k=\ceil{(1+k/8)n}$ for $k\geq 0$, and observe that for $r$ larger than some universal constant, if $n\leq\sigma\leq N$ and $X$ has no cut time in $[\sigma-n/4,\sigma]$, then $X$ has no cut times in at least one of the intervals belonging to the family $\{[T_k-T_k/(\log T_k)^6,T_k]:0\leq k\leq 8\ceil{(\log n)^{1/3}}\}$. Therefore, for $r$ larger than some universal  constant, we have that
	\begin{equation}\label{eq:Cprime_bound}
		C_r^\prime\leq \sum_{k=0}^{8\ceil{(\log n)^{1/3}}} \sum_{w\in \eta\cap\Lambda(x,r)} \sum_{j=n+1}^{N}  \mu(w)\mathbbm{1}\Big(X_j=w,\text{$X$ has no cut time in $[T_k-T_k/(\log T_k)^6,T_k]$}\Big).
	\end{equation}
We also have by symmetry that
  \begin{multline*}
\bP_x\Big(X_j=y,\text{$X$ has no cut time in $[T_k-T_k/(\log T_k)^6,T_k]$}\Big) \\= \bP_y\Big(X_j=x,\text{$X$ has no cut time in $[T_k-T_k/(\log T_k)^6,T_k]$}\Big) 
  \end{multline*}
  for each $x,y\in \Z^4$ and $j\geq 0$, so that for $r$ larger than some universal constant
	\begin{multline}\mathbf{E} C_r^\prime\leq \frac{N}{\abs{\Lambda(3x,r)}}\sum_{k=0}^{8\ceil{(\log n)^{1/3}}}   \pb_0\Big(\text{$X$ has no cut time in $[T_k,T_k-T_k/(\log T_k)^6]$}\Big)\\ 
		\preceq \lambda r^{-2}\sum_{k=0}^{8\ceil{(\log n)^{1/3}}} \frac{\log\log T_k}{\log T_k}\preceq \lambda r^{-2}(\log n)^{1/3}\frac{\log\log n}{\log n}=o(r^{-2}),
	\end{multline}
where the second inequality follows from Lemma \ref{lem:Lawlercuttimes}.
	We have now shown \eqref{eq:lower_bound_exp_A}, and so to complete  the proof we must show \eqref{eq:upper_bound_exp_A}, which upper bounds the second moment of $A$. 

\medskip

\noindent
	\textbf{Upper bounding the second moment of $A$.} It is at this stage of the proof that we benefit from defining $A$ in terms of the measure $\mu$.
Indeed, we can use the Markov property to compute that
	\begin{align*}
		\mathbf{E}{A_r^2}&\leq\mathbf{E}\left[\left(\sum_{i\geq0}\sum_{ w\in\eta\cap\Lambda(x,r)}\mu(w)\mathbbm{1}\big(X_i=w\big)\right)^2\right]
		\\&\leq 2\mathbf{E}\left[\sum_{i\geq 0}\sum_{w,z\in\eta\cap\Lambda(x,r)}\mu(w)\mu(z)\mathbbm{1}\big(X_i=w\big)\sum_{j\geq i}\mathbbm{1}\big(X_j=z\big)\right]\\
		&\asymp 2\mathbf{E}\left[\sum_{i\geq0}\sum_{ w,z\in\eta\cap\Lambda(x,r)}\mu(w)\mu(z)G(w,z)\mathbbm{1}\big(X_i=w\big)\right]\\
    &=\frac{2}{\abs{\Lambda(x,3r)}}\sum_{w,z\in\eta\cap\Lambda(x,r)}\mu(w)\mu(z)G(w,z)\sum_{i\geq 0}\sum_{v\in\Lambda(x,3r)}\bP_v\big(X_i=w\big),
  \end{align*}
  and hence by time-reversal that
  \begin{align*}
  \mathbf{E}{A_r^2}
		&\preceq r^{-4}
		\sum_{w,z\in\eta\cap\Lambda(x,r)}\mu(w)\mu(z)G(w,z)\sum_{i\geq 0}\mathbf{P}_w\big(X_i\in \Lambda(x,3r)\big)\\
		&\preceq r^{-2} \sum_{w,z\in\eta\cap\Lambda(x,r)}\mu(w)\mu(z)G(w,z)=r^{-2} \mathrm{Cap}^{-1}(\eta\cap \Lambda(x,r)),
	\end{align*}
   where the final inequality follows since the random walk spends at most $O(r^2)$ time in any ball of radius $r$ in expectation (which follows from the Green's function bound $G(x,y)\preceq \norm{x-y}_2^{-2}$ for $x\neq y$),
     and the final equality follows from the definition of $\mu$. This concludes the proof of \eqref{eq:upper_bound_exp_A} and hence the proof of the proposition. \qedhere
\end{proof}

We now turn to the proof of the variance estimate of Proposition~\ref{prop:upper_bound_variance}.
We will require the following lemma relating the capacity of a set $S$ to the probability that a random walk, started at a uniform position in a ball containing $S$, hits $S$. The lemma will follow straightforwardly from \cite[Theorem 2.2]{MR1349175} and Lemma \ref{lem:capacity_formulation}. We prove the result in all dimensions $d\geq 3$ for completeness; the implicit constants may depend on $d$.
\begin{lemma} \label{lem:uniform_point_capacity}
	Fix a dimension $d\geq 3$, a radius $r\geq 1$, and let $S\subseteq\Lambda(r):=\{x\in\Z^d:\norm{x}_\infty\leq r\}$. Let $X$ be a simple random walk on $\Z^d$. Then 
	\[
	\sum_{x\in\Lambda(r)}\pb_x(X \ \mathrm{hits}\ S)\asymp r^{2}\mathrm{Cap}(S).
	\]
\end{lemma}

\begin{proof}[Proof of Lemma~\ref{lem:uniform_point_capacity}]
	\cite[Theorem 2.2]{MR1349175} states that for any transient Markov chain $(X_n)_{n\geq 0}$ on a countable state space $\Omega$ with initial state $\rho$ and Green's function\footnote{Note that, unlike in the rest of the paper, the Green's function does not include normalization by $\deg(y)^{-1}$, such normalization being unappropriate for non-reversible chains. Since all our estimates hold only to within multiplicative constants, the distinction is not important.} $G(x,y)=\sum_{n\geq 0}\pb_x(X_n=y)$, we have that
	\[
	\pr_\rho(\text{$X$ hits $S$})\asymp\inf_{\mu} \left[\sum_{x,y\in S}\mu(x) \frac{G(x,y)}{G(\rho,y)}\mu(y)\right]^{-1},
	\]
	for any subset $S\subseteq \Omega$, where the infimum on the right hand side is taken over probability measures on $S$. We would like to apply this result with $X$ a simple random walk on state space $\Z^d$, however, we would like the walk to start at a random vertex. To achieve this, we attach a `ghost vertex' to the state space from which the random walk will start. We set up the transition probabilities from the ghost vertex so that after one step, the walk's distribution on $\Z^d$ is equal to that which we desire.
	
	Define the set $\Z^d_*=\Z^d\cup\{*\}$, where $*$ is the additional ghost vertex, and define the Markov transition kernel $p$ on the state space $S$ by $p(x,y)=\frac{1}{8}\mathbbm{1}(x\sim y)$ for $x,y\in\Z^d$ and $p(*,z)=1/\abs{\Lambda}$ for $z\in\Lambda:=\Lambda(r)$. 
	Note that a trajectory of this chain, which we will denote by $X$, is just a simple random walk on $\Z^d$ when started in $\Z^d$. We observe that
	\begin{equation}\label{eq:KCap}
	\frac{1}{\abs{\Lambda}}\sum_{x\in\Lambda}\pb_x(X \ \mathrm{hits}\ S)=\pb_*(X \ \mathrm{hits}\ S)\asymp \inf_{\mu} \left[\sum_{x,y\in S}\mu(x) \frac{G(x,y)}{G(*,y)}\mu(y)\right]^{-1},
	\end{equation}
	for any subset $S\subseteq \Lambda$.
	An integral comparison yields that 
	\[
	G(*,y)\asymp\frac{1}{\abs{\Lambda}}\sum_{x\in\Lambda} \frac{1}{(1\vee\norm{x-y}_\infty)^{d-2}}\asymp r^{2-d},
	\]	
	for $y\in\Lambda$, and so
	\[\inf_{\mu} \left[\sum_{x,y\in S}\mu(x) \frac{G(x,y)}{G(*,y)}\mu(y)\right]^{-1}\asymp r^{2-d}\inf_{\mu} \left[\sum_{x,y\in S}\mu(x) G(x,y)\mu(y)\right]^{-1}.
	\]
	Substituting this into \eqref{eq:KCap} and applying Lemma \ref{lem:capacity_formulation}, we get
	\[
	\sum_{x\in\Lambda(r)}\pb_x(X \ \mathrm{hits}\ S)\asymp r^{2}\inf_{\mu} \left[\sum_{x,y\in S}\mu(x) G(x,y)\mu(y)\right]^{-1}\asymp r^{2}\mathrm{Cap}(S)
	\]
  as claimed. (We do not have an exact equality on the right hand side because we are using a slightly different definition of the Green's function than usual.)
\end{proof}

\begin{proof}[Proof of Proposition \ref{prop:upper_bound_variance}]
Given $y,z\in \Z^4$,
	let $Y$ be a random walk started at $y$  and let $Z$ be an independent random walk started at $z$ and write $\bP_{y,z}$ for the joint law of $Y$ and $Z$. Let $\sigma_1$ be the first time $Y$ hits $\eta$ and let $\sigma_2$ be the first time $Z$ hits $\eta\cup\LE(Y^{\sigma_1})$. 
	We continue to write $n=\lceil r^2 (\log r)^{-1/3}\rceil$ as in the previous proof.
	Abbreviating $M=M_\alpha$, $\Lambda=\Lambda(x,3r)$, we have by Wilson's algorithm that
	\begin{equation}\label{eq:expand_the_square}
		\begin{aligned}
			\mathbb{E}^\eta[\abs{M(x,r)}^2]&\leq\sum_{{y,z}\in \Lambda}\pr^\eta\!\big(y,z\in M(x,r)\big)\\
			&\leq	\sum_{{y,z}\in \Lambda} \pb_{y,z}(\sigma_1<\infty,\sigma_2<\infty,\abs{\LE(Y^{\sigma_1})}\leq n,\abs{\LE(Z^{\sigma_2})}\leq n,\\
			&\hspace{2cm}\LE(Y^{\sigma_1})\subseteq \Lambda, \LE(Z^{\sigma_2})\subseteq \Lambda,\LE(Y^{\sigma_1}),\LE(Z^{\sigma_2}) \text{ both good}).
		\end{aligned}
	\end{equation}
	Now, on the event that $\sigma_1,\sigma_2<\infty$, let $\sigma_3$ be the time $Z$ first hits $\LE(Y^{\sigma_1})$ and let $\sigma_4$ be the time $Z$ first hits $\eta$.  We split according to whether $\sigma_3\leq \sigma_4$ or $\sigma_4<\sigma_3$, beginning with the case $\sigma_4<\sigma_3$. Observing that $\sigma_2=\sigma_4$ on this event, we obtain
	\begin{align}\label{eq:exp_squared}
	&\sum_{{y,z}\in \Lambda} \pb_{y,z}(\sigma_1<\infty,\,\sigma_2<\infty,\,\abs{\LE(Y^{\sigma_1})}\leq n,\,\abs{\LE(Z^{\sigma_2})}\leq n,
  \nonumber\\
&\phantom{\sum_{{y,z}\in \Lambda}}\hspace{4cm}\LE(Y^{\sigma_1})\subseteq \Lambda,\, \LE(Z^{\sigma_2})\subseteq \Lambda,\LE(Y^{\sigma_1}),\,\LE(Z^{\sigma_2}) \text{ both good},\text{ and }\sigma_4<\sigma_3).
  \nonumber\\
	&\phantom{\sum_{{y,z}\in \Lambda}}\leq \sum_{{y,z}\in \Lambda} \pb_{y,z}(\sigma_1<\infty,\,\sigma_4<\infty,\,\abs{\LE(Y^{\sigma_1})}\leq n,\,\abs{\LE(Z^{\sigma_4})}\leq n,
  \nonumber\\
	&\phantom{\sum_{{y,z}\in \Lambda}}\hspace{4cm}\LE(Y^{\sigma_1})\subseteq \Lambda,\, \LE(Z^{\sigma_4})\subseteq \Lambda, \text{ and }\LE(Y^{\sigma_1}),\LE(Z^{\sigma_4}) \text{ both good}).
	\nonumber\\
	&\phantom{\sum_{{y,z}\in \Lambda}}=\sum_{{y,z}\in \Lambda} \pb_{y}(\sigma_1<\infty,\,\abs{\LE(Y^{\sigma_1})}\leq n,\,\LE(Y^{\sigma_1})\subseteq \Lambda, \text{ and } \LE(Y^{\sigma_1})\text{ good})
  \nonumber\\
 	&\phantom{\sum_{{y,z}\in \Lambda}}\hspace{3.75cm}\cdot\pb_{z}(\sigma_4<\infty,\,\abs{\LE(Y^{\sigma_4})}\leq n,\,\LE(Y^{\sigma_4})\subseteq \Lambda,\,\LE(Y^{\sigma_4})\text{ good})
  \nonumber\\
 	&\phantom{\sum_{{y,z}\in \Lambda}}=\left[\sum_{y\in \Lambda}\pb_{y}(\sigma_1<\infty,\abs{\LE(Y^{\sigma_1})}\leq n ,\LE(Y^{\sigma_1})\subseteq \Lambda,\LE(Y^{\sigma_1})\text{ good})\right]^2
   =\mathbb{E}^\eta[\abs{M(y,r)}]^2,
	\end{align}
where the first equality follows by independence of $Y$ and $Z$ conditional on $\eta$, and the last follows by an application of Wilson's algorithm.
On the other hand, 
 if $\sigma_3\leq\sigma_4$ then $\sigma_2=\sigma_3$, and so we get
\begin{align}\label{eq:sq_remainder}
	&\sum_{{y,z}\in \Lambda} \pb_{y,z}(\sigma_1<\infty,\sigma_2<\infty,\abs{\LE(Y^{\sigma_1})}\leq n,\abs{\LE(Z^{\sigma_2})}\leq n,
  \nonumber\\
	&\hspace{5.5cm}\LE(Y^{\sigma_1})\subseteq \Lambda, \LE(Z^{\sigma_2})\subseteq \Lambda,\LE(Y^{\sigma_1}),\LE(Z^{\sigma_2}) \text{ both good},\sigma_3\leq\sigma_4)
  \nonumber\\
	&\hspace{1cm}\leq\sum_{y\in \Lambda} \mathbf{E}_y\Bigg[\mathbbm{1}(\sigma_1<\infty,\abs{\LE(Y^{\sigma_1})}\leq n,\LE(Y^{\sigma_1})\subseteq \Lambda,\LE(Y^{\sigma_1}) \text{ good}) \sum_{z\in \Lambda} \pb_{y,z}(\sigma_3<\infty \mid Y) \Bigg]
  \nonumber\\
		&\hspace{1cm}\leq \alpha \frac{r^4}{\log r}	\sum_{y\in \Lambda}\pb_{y}(\sigma_1<\infty,\abs{\LE(Y^{\sigma_1})}\leq n ,\LE(Y^{\sigma_1})\subseteq \Lambda,\LE(Y^{\sigma_1})\text{ good})
	\nonumber\\
  &\hspace{1cm}= \alpha \frac{r^4}{\log r}\mathbb{E}^\eta\abs{M(x,r)},
\end{align}
where the final inequality follows by the definition of `good', and the final equality follows by an application of Wilson's algorithm.
Substituting \eqref{eq:sq_remainder} and \eqref{eq:exp_squared} into \eqref{eq:expand_the_square} with a union bound yields
	\begin{equation*}
		\mathbb{E}^\eta[\abs{M(x,r)}^2]\leq \mathbb{E}^\eta[\abs{M(x,r)}]^2+\alpha \frac{r^4}{\log r}\mathbb{E}^\eta\abs{M(x,r)}
	\end{equation*}
and hence that
    \begin{equation}\label{eq:divide_using_good}
    \operatorname{Var}^\eta(\abs{M(x,r)})\leq \alpha \frac{r^4}{\log r}\mathbb{E}^\eta\abs{M(x,r)}.
  \end{equation}
Finally we upper bound $\mathbb{E}^\eta\abs{M(x,r)}$. We have that
	\[\mathbb{E}^\eta |M(x,r)|\leq \sum_{y\in \Lambda} \pb_y(X \text{ hits } \eta\cap\Lambda),
	\]
so that applying Lemma \ref{lem:uniform_point_capacity} to the right hand side and plugging the resulting inequality into \eqref{eq:divide_using_good} concludes the proof.
\end{proof}

Our next goal is to deduce Proposition~\ref{prop:volume_lower} from Propositions~\ref{prop:lower_bound_expectation} and \ref{prop:upper_bound_variance}.
To proceed we will need the following result controlling the capacity of the first $n$ steps of a loop-erased random walk which follows easily from \cite[Proposition 3.4]{hutchcroft2020logarithmic}\footnote{As pointed out to us by the referee, the proof of \cite[Proposition 3.4]{hutchcroft2020logarithmic} contains a minor error, which we now explain how to correct. In the sentence beginning ``Using the trivial bound $\operatorname{Var}(\operatorname{Cap}(\mathsf{LE}(X[I_k]))\leq |I_k|^2$ \ldots'', this trivial bound is not actually strong enough to deduce the next stated bound using Chebyshev's inequality, and in fact the bound one does obtain via this method is too weak to be used to prove the proposition. To fix this, one can instead bound $\operatorname{Var}(\operatorname{Cap}(\mathsf{LE}(X[I_k]))$ by the second moment of the capacity of the random walk $X[I_k]$, which is bounded in \cite[Lemma 2.5]{hutchcroft2020logarithmic} and Theorem \ref{theorem:LERWconc}}.
\begin{proposition}\label{prop:cap_bound}
	Let $X$ be a random walk on $\Z^4$ started at the origin. There exists a constant $C>0$ such that we have
	\[
	\pb\left(\mathrm{Cap}(\LE(X)^n)\leq \frac{Cn}{(\log n)^{2/3}}\right)\preceq \frac{\log\log n}{(\log n)^{2/3}},
	\]
	for every $n\geq 2$.
\end{proposition}
\begin{proof}
	By \cite[Proposition 3.4]{hutchcroft2020logarithmic}, we know that there exists a constant $c$ such that
	\[
	\pb\left(\textrm{Cap}(\LE_\infty(X^n))\leq \frac{cn}{\log n}\right)\preceq \frac{1}{(\log n)^{2/3}},
	\]
	for each $n\geq 2$.
	Fix $\epsilon\in(0,1/3)$. Employing a union bound and the fact that capacity is increasing, we obtain 
	\begin{align*}
		\pb\left(\textrm{Cap}(\LE(X)^n)\leq \frac{C n}{(\log n)^{2/3}}\right)&=\pb\left(\textrm{Cap}(\LE_{\infty}(X^{\ell_n}))\leq \frac{C n}{(\log n)^{2/3}}\right)\\ 
		&\leq\pb\left(\textrm{Cap}(\LE_{\infty}(X^{(1-\epsilon)n(\log n)^{1/3}}))\leq \frac{C n}{(\log n)^{2/3}}\right)
		+\pb\left(\left\vert\frac{\ell_n}{n(\log n)^{1/3}}-1\right\vert>\epsilon\right)\\
		&\preceq \frac{1}{(\log n)^{2/3}}+\frac{\log\log n}{(\log n)^{2/3}}\preceq\frac{\log\log n}{(\log n)^{2/3}}	
	\end{align*}
	when we choose $C<c(1-\epsilon)$.
\end{proof}

We will also use the following covering lemma, whose proof we defer to the end of the section.

\begin{lemma}\label{lemma:disjoint_balls_capacity}
	Let  $S$ be a finite subset of $\Z^4$, and let $r\geq 1$.
	Then there exists an integer $K$ and points $\{x_i:1\leq i\leq K\}\subseteq\Z^4$ such that the balls $\Lambda(x_i,3r)$ are disjoint, $\{x_i\}_{1\leq i\leq K}\subseteq S+\Lambda(r)$, and
	\begin{itemize}
		\item $\sum_{i=1}^K \mathrm{Cap}(S\cap \Lambda(x_i,r))\geq 5^{-4}\mathrm{Cap}(S)$, and
		\item $\sum_{i=1}^K \mathrm{Cap}(S\cap \Lambda(x_i,3r))\leq 21^4\sum_{i=1}^K \mathrm{Cap}(S\cap \Lambda(x_i,r))$.
	\end{itemize}
\end{lemma}

We now have everything we need to complete the proof of Proposition~\ref{prop:volume_lower} given Lemma~\ref{lemma:disjoint_balls_capacity}.

\begin{proof}[Proof of Proposition~\ref{prop:volume_lower}]
Let $\alpha_0,r_0$ be the constants yielded by Proposition \ref{prop:lower_bound_expectation}, and fix $r\geq r_0\vee 2$, $\alpha>\alpha_0$. For the remainder of the proof we will abbreviate $M=M_\alpha$.
Let $K\geq 1$ and suppose that $\{x_i:1\leq i \leq K\}\subseteq\Z^4$ is a set of points such that the family of boxes $(\Lambda(x_i,3r))_{i=1}^K$ are mutually disjoint. We first show that the random variables $\abs{M(x_i,r)}$ are pairwise negatively correlated conditional on $\eta$ in the sense that
\[
\mathbb{E}^\eta \Bigl[\abs{M(x_i,r)} \cdot \abs{M(x_j,r)}\Bigr] \leq \mathbb{E}^\eta\Bigl[ \abs{M(x_i,r)}\Bigr] \mathbb{E}^\eta \Bigl[\abs{M(x_j,r)}\Bigr]
\]
for every $1\leq i < j \leq K$.
 Indeed, suppose that $u\in \Lambda(x_i,3r)$ and $v\in \Lambda(x_j,3r)$ for some $i\neq j$. We sample the UST conditional on $\eta=\Gamma(0,\infty)$ with Wilson's algorithm, beginning with a random walk $X$ started at $u$, followed by another walk $Y$ started at $v$. 
Let $\tau_1$ be the first time $X$ hits $\eta$, let $\tau_2$ be the first time $Y$ hits $\LE(X^{\tau_1})\cup\eta$, and 
 let $\tau_2^\prime$ be the first time $Y$ hits $\eta$. Then
\begin{multline*}
	\pr^\eta\bigl(u\in M(x_i,r),\, v \in M(x_j,r)\bigr)
	=\pr^\eta(u\in M(x_i,r))\\
  \cdot \pr^\eta\left(\LE(Y^{\tau_2})\subseteq \Lambda(x_j,3r),\, 
	\abs{\LE(Y^{\tau_2})}\leq \frac{r^2}{(\log r)^{1/3}},\, \LE(Y^{\tau_2}) \text{ is $(\alpha,r)$-good}\ \Big\vert\ u\in M(x_i,r)\right).
\end{multline*}
We have by the definition of $M(x_i,r)$ that if $u\in M(x_i,r)$ then $\LE(X^{\tau_1})\subseteq\Lambda(x_i,3r)$, so that if $\LE(Y^{\tau_2})\subseteq \Lambda(x_j,3r)$ then $\tau_2=\tau_2^\prime$. It follows that
\begin{align*}
	&\pr^\eta\bigl(v\in M(x_j,r)\mid u\in M(x_i,r)\big)\\
	&\hspace{2cm}\leq \pr^\eta\left(\LE(Y^{\tau_2^\prime})\subseteq \Lambda(x_j,3r),\, 
	\abs{\LE(Y^{\tau_2^\prime})}\leq \frac{r^2}{(\log r)^{1/3}},\, \LE(Y^{\tau_2^\prime}) \text{ is $(\alpha,r)$-good}\ \Big\vert\ u\in M(x_i,r)\right)\\
	&\hspace{2cm}=\pr^\eta\left(\LE(Y^{\tau_2^\prime})\subseteq \Lambda(x_j,3r),\, 
	\abs{\LE(Y^{\tau_2^\prime})}\leq \frac{r^2}{(\log r)^{1/3}},\, \LE(Y^{\tau_2^\prime}) \text{ is $(\alpha,r)$-good}\right)\\
	&\hspace{2cm}=\pr^\eta\big(v\in M(x_j,r)\big)
	\end{align*}
where the first equality follows because $Y^{\tau_2^\prime}$ is independent from the event $\{u\in M(x_i,r)\}$ conditional on $\eta$ and where the last equality follows by an application of Wilson's algorithm. The claimed negative correlation of $\abs{M(x_i,r)}$ and $\abs{M(x_j,r)}$ follows by summing over $u$ and $v$.
%
%
Negativity of the correlations immediately implies that
\[
\textrm{Var}^\eta\left(\Bigl|\bigcup_{i=1}^K M(x_i,r)\Bigr|\right)\leq \sum_{1\leq i\leq K} \textrm{Var}^\eta(\abs{M(x_i,r)}),
\]
and we deduce by Chebyshev together with Propositions~\ref{prop:lower_bound_expectation} and \ref{prop:upper_bound_variance} that
\begin{equation} \label{eq:Cheb}
	\pr^\eta\left(\Bigl|\bigcup_{i=1}^K M(x_i,r)\Bigr|\leq c_1r^{2}\sum_{i=1}^K \textrm{Cap}(\eta \cap \Lambda(x_i,r))\right)\preceq \frac{r^2}{\log r} \cdot \frac{\sum_{i=1}^K \mathrm{Cap}(\eta\cap \Lambda(x_i,3r))}{\left(\sum_{i=1}^K \textrm{Cap}\big(\eta \cap \Lambda(x_i,r)\big			)\right)^2}\
\end{equation}
for some constant $c_1>0$. Note that this estimate holds for any $K\geq 1$ and any collection of points $(x_i)_{i=1}^K$ in $\Z^4$ such that the family of boxes $(\Lambda(x_i,3r))_{i=1}^K$ are mutually disjoint, where we are free to choose $K$ and $(x_i)_{i=1}^K$ as functions of $\eta$ if we wish. (Of course the points we choose must be conditionally independent of the rest of the UST given $\eta$.)

We now want to apply this estimate to prove our lower tail estimate on $|\fB(n)|$. Fix $n\geq 1$, and for each $R\geq 1$, let $\mathscr{A}_R$ be the event that $\|\eta_i\|_\infty \geq 2R$ for every $i\geq n/2$. Observe from the definitions that if $\mathscr{A}_R$ holds and $r\geq 2$ is such that $r^2(\log r)^{-1/3} \leq n/2$ and $3r\leq R$ then 
\[
|\fB(n)| \geq \Bigl|\bigcup_{i=1}^K M(x_i,r)\Bigr|
\]
for any collection of points $x_1,\ldots,x_K$ in $\Lambda(R)$: the definition of the set $M(x_i,r)$ and the choice of $r$ ensures the path connecting $x$ to $\eta$ is contained in $\Lambda(2R)$ and has length at most $n/2$, while the definition of $\mathscr{A}_R$ ensures that this path meets $\eta$ within the first $n/2$ steps of $\eta$.
Thus, choosing these points as a function of $\eta$ and $r\geq 1$ as in the covering lemma, Lemma~\ref{lemma:disjoint_balls_capacity}, where we take $S = \eta \cap \Lambda(R)$, we deduce from \eqref{eq:Cheb} that there exists a constant $c_1$ such that
\begin{equation}
  \mathbbm{1}(\mathscr{A}_R)\pr^\eta\left(|\mathfrak{B}(n)|\leq c_1 r^{2}\textrm{Cap}(\eta \cap \Lambda(R))\right) \preceq \frac{r^2}{\log r} \cdot \frac{1}{\textrm{Cap}(\eta \cap \Lambda(R))}
\end{equation}
for every $r,R\geq 2$ such that $r^2(\log r)^{-1/3} \leq n/2$ and $3r\leq R$. As such, we have by a union bound that
\begin{equation}
\label{eq:optimize_three_parameters}
\pr\left(|\mathfrak{B}(n)|\leq \frac{c_1 r^{2} R^2}{\lambda \log R}\right) \preceq \frac{\lambda r^2 \log R}{R^2 (\log r)} + \P(\mathscr{A}_R^c) + \P\left(\textrm{Cap}(\eta \cap\Lambda(R)) \leq \frac{R^2}{\lambda \log R}\right)
\end{equation}
for every $r,R\geq 2$ such that $r^2(\log r)^{-1/3} \leq n/2$ and $3r\leq R$ and every $\lambda \geq 1$. 

To proceed, we will bound the second and third terms on the right hand side then optimize over the choice of $r$, $R$, and $\lambda$.
To bound $\P(\mathscr{A}_R)$, we use Wilson's algorithm to write
\begin{align*}
		\P(\mathscr{A}_R^c)&=\bP_0(\LE(X)_i \in \Lambda(2R) \text{ for some $i\geq n/2$}) 
		\\&\leq 
		\bP_0\left(\ell_{\lfloor n/2 \rfloor}(X) \leq \frac{1}{4}n(\log n)^{1/3} \right)
		+\bP_0\left(X_j \in \Lambda(2R) \text{ for some $j\geq \frac{1}{4}n(\log n)^{1/3}$}\right)\\
		&\preceq \frac{\log \log n}{(\log n)^{2/3}}+ \frac{R^2 }{n(\log n)^{1/3}},
\end{align*}
where the first term has been bounded using Theorem~\ref{theorem:LERWconc} and the second follows by a standard random walk computation (for example, it follows by \cite[Lemma 4.4]{MR4055195} and Markov's inequality). To bound the second term, 
we use the union bound
\[
\P\left(\textrm{Cap}(\eta \cap\Lambda(R)) \leq \frac{R^2}{\lambda \log R}\right) \leq \P(\|\eta_i\|\geq R \text{ for some $i\leq k$}) + \P\left(\textrm{Cap}(\eta^k) \leq \frac{R^2}{\lambda \log R}\right)
\]
for every $R,k\geq 1$ and $\lambda\geq 1$. Using Wilson's algorithm and a further union bound yields that
\begin{multline*}
\P\left(\textrm{Cap}(\eta \cap\Lambda(R)) \leq \frac{R^2}{\lambda \log R}\right) \leq 
\bP_0\left(\ell_k \geq 2 k (\log k)^{1/3} \right) +
\bP_0(\|X_j\|\geq R \text{ for some $j\leq 2 k (\log k)^{1/3}$})\\ + \bP_0\left(\textrm{Cap}(\LE(X)^k) \leq \frac{R^2}{\lambda \log R}\right),
\end{multline*}
and we deduce from Theorem~\ref{theorem:LERWconc}, the maximal version of Azuma-Hoeffding \cite[Section 2]{McDiarmid1998}, and Proposition~\ref{prop:cap_bound} that there exists a positive constant $C$ such that
\begin{equation*}
\P\left(\textrm{Cap}(\eta \cap\Lambda(R)) \leq \frac{R^2}{\lambda \log R}\right) \preceq \frac{\log \log k}{(\log k)^{2/3}} +
\exp\left[-\Omega\left(\frac{R^2}{k(\log k)^{1/3}}\right)\right]
\end{equation*}
for every $R,k\geq 1$ such that $k(\log k)^{-2/3} \leq C \lambda^{-1}  R^2 (\log R)^{-1}$. If $\lambda \leq R^{1/2}$ then the maximal such $k$ is of order $\lambda^{-1} R^2 (\log R)^{-1/3}$ and it follows by calculus that
\begin{equation*}
\P\left(\textrm{Cap}(\eta \cap\Lambda(R)) \leq \frac{R^2}{\lambda \log R}\right) \preceq \frac{\log \log R}{(\log R)^{2/3}} + \exp\left[-\Omega(\lambda^{-1})\right]
\end{equation*}
for every $R\geq 3$ and $1\leq \lambda \leq R^{1/2}$. Putting these estimates together yields that
\[
\pr\left(|\mathfrak{B}(n)|\leq \frac{c_1 r^{2} R^2}{\lambda \log R}\right) \leq \frac{\lambda r^2 \log R}{R^2\log r} + \frac{\log \log n}{(\log n)^{2/3}} + \frac{R^2}{n(\log n)^{1/3}} + \frac{\log \log R}{(\log R)^{2/3}} + \exp\left[-\Omega(\lambda^{-1})\right]
\]
for every $r,R\geq 2$ such that $r^2(\log r)^{-1/3} \leq n/2$ and $3r\leq R$ and every $1\leq \lambda \leq R^{1/2}$. Letting $\beta \geq 10$, taking $R=\lceil \beta^{-1} n^{1/2} (\log n)^{1/6}\rceil$, $r= \lceil \beta^{-2} n^{1/2} (\log n)^{1/6}\rceil$ and $\lambda =\beta$ yields that if $n \geq \beta^4$ then
\begin{align*}
\pr\left(|\mathfrak{B}(n)|\leq \frac{c_2 n^2}{\beta^5 (\log n)^{1/3}}\right) \preceq \beta^{-1} + \frac{\log \log n}{(\log n)^{2/3}} + \beta^{-2} + \frac{\log \log n}{(\log n)^{2/3}} + \exp\left[-\Omega(\beta^{-1})\right]
&\preceq \beta^{-1} + \frac{\log \log n}{(\log n)^{2/3}},
\end{align*}
which implies the claim.
\end{proof}

It remains to prove our covering lemma for the capacity, Lemma~\ref{lemma:disjoint_balls_capacity}. The proof, which exhibits and analyzes a greedy algorithm for constructing the desired set of balls, follows a standard strategy for proving covering lemmas of similar form.

\begin{proof}[Proof of Lemma \ref{lemma:disjoint_balls_capacity}]
Consider the set of centres $\mathcal{C}=\{x\in(2r+1)\Z^4:\mathrm{Cap}(\Lambda(x,r)\cap S )>0\}$ and the partition of $\Z^4$ defined by $\mathcal{B}=\{\Lambda(x,r):x\in \mathcal{C}\}$. Note that  $x\in S +\Lambda(r)$ for each $x\in\mathcal{C}$, since otherwise the box $\Lambda(x,r)$ would not contain any points of $ S $. Given $x\in\mathcal{C}$, we write $A[x]=\{y\in\mathcal{C}:\norm{y-x}_\infty\leq 2r+1\}$ for the set of centres in $\mathcal{C}$ equal to or adjacent to $x$. We note the crude bound $\#{A[x]}\leq 3^4$. Similarly, we write $A^2[x]=\{y\in\mathcal{C}:\norm{y-x}_\infty\leq 4r+2\}$, $A^3[x]=\{y\in\mathcal{C}:\norm{y-x}_\infty\leq 6r+3\}$, and note that $\#{A^2[x]}\leq 5^4$, $\#{A^3[x]}\leq 7^4$.

We will construct the sequence $(x_i)_{i=1}^K$ using a greedy algorithm. By subadditivity of capacity (which is an immediate consequence of the variational principle of Lemma \ref{lem:capacity_formulation}), we know that 
\begin{equation}\label{eq:subadd1}
	\Pi:=\sum_{x\in\mathcal{C}}\mathrm{Cap}( S \cap \Lambda(x,r))\geq \mathrm{Cap}( S ).
\end{equation}
Define the list of centres $(x_i)_{i\geq 0}\subseteq \mathcal{C}$ as follows. Let $\mathcal{C}_0=\mathcal{C}$, and for $i\geq 0$ such that $\mathcal{C}_i\neq\emptyset$, let
\[
x_i=\argmax\{\mathrm{Cap}(\Lambda(x,r)\cap S ):x\in\mathcal{C}_i\};\qquad \mathcal{C}_{i+1} = \mathcal{C}_i\setminus A^2[x_i],
\]
Write $I=\inf\{i\geq 0: \mathcal{C}_i=\emptyset\}$ and define $\kappa_i=\mathrm{Cap}(\Lambda(x_i,r)\cap S )$ for $0\leq i< I$.  We claim that 
\begin{equation}
\sum_{0\leq i\leq n}\mathrm{Cap}( S \cap \Lambda(x_i,3r))\leq 21^4 \sum_{0\leq j\leq n} \kappa_j \qquad \text{for every $n<I$.}
\label{eq:covering_claim}
\end{equation}
Fix $0\leq i<I$. We note that for any $y\in A^2[x_i]$, there exists a unique $0\leq j\leq i$ such that $y\in \mathcal{C}_{j}\setminus \mathcal{C}_{j+1}$. By definition of $\kappa_j$ and $x_j$, it must then hold that $\mathrm{Cap}(\Lambda(y,r)\cap S )\leq\kappa_j$. By subadditivity of capacity, we can therefore write
\[
\mathrm{Cap}(\Lambda(x_i,3r)\cap S )\leq \sum_{y\in A[x_i]}\mathrm{Cap}(\Lambda(y,r)\cap S  )\leq  \sum_{y\in A[x_i]}\sum_{j\leq i}\kappa_{j}\mathbbm{1}(y\in\mathcal{C}_j\setminus\mathcal{C}_{j+1}).
\]
Observing that $\mathcal{C}_j\setminus\mathcal{C}_{j+1}\subseteq A^2[x_j]$ for $j<I$, we get
\[
\mathrm{Cap}(\Lambda(x_i,3r)\cap S )\leq  \sum_{j\leq i}\kappa_{j}\abs{A[x_i]\cap A^2[x_j]}.
\]
By switching the order of summation, we have
\[
\sum_{0\leq i\leq n}\mathrm{Cap}( \Lambda(x_i,3r)\cap S )\leq  \sum_{0\leq j\leq n} \kappa_{j}\sum_{j\leq i\leq n}\abs{A[x_i]\cap A^2[x_j]}.
\]
Finally, $\abs{A[x_i]\cap A^2[x_j]}\leq\abs{A[x_i]}\leq 3^4$, and if $\abs{A[x_i]\cap A^2[x_j]}\neq 0$, then $x_i\in A^3[x_j]$. The $x_i$ are all distinct, and there are at most $7^4$ elements in $A^3[x_j]$, and so the summations over $i$ on the right hand side are bounded above by $3^4\times 7^4=21^4$, thus proving the claim \eqref{eq:covering_claim}.

Next, observe that for $i\geq 0$ 
\[
\sum_{x\in \mathcal{C}_i}\mathrm{Cap} (\Lambda(x,r)\cap S ) \geq \Pi-5^4 \sum_{0\leq j\leq i-1}\kappa_j.
\]
Indeed, at stage $i$ in the algorithm we remove at most $5^4$ centres from $\mathcal{C}_i$ to give $\mathcal{C}_{i+1}$, and for each of these centres $x$, we must have $\mathrm{Cap}( S \cap \Lambda(x,r))\leq \kappa_i$. Putting $i=I$ in the above equation gives
\[5^4 \sum_{0\leq j<I}\kappa_j\geq \Pi,\]
and so by \eqref{eq:subadd1}, we have $\sum_{0\leq j<I}\kappa_j\geq 5^{-4} \mathrm{Cap}( S )$. Finally, we note that for $0\leq i<j<I$, by construction $x_j\notin A^2[x_i]$, and so the balls $\Lambda(x_i,3r)$ and $\Lambda(x_j,3r)$ are disjoint.
\end{proof}

\begin{remark}
Note that the proof of Lemma~\ref{lemma:disjoint_balls_capacity} does not use any properties of the capacity other than subadditivity and non-negativity, so that a similar covering lemma holds for any subadditive, non-negative set function.
\end{remark}

\section{Random walk}

We now apply our main geometric theorem, Theorem~\ref{thm:main_geometric_theorem}, to study the behaviour of the random walk on the 4d UST. We begin by applying our results together with those of \cite{hutchcroft2020logarithmic} to prove our effective resistance estimate, Theorem~\ref{cor:eff_cond}, in Section~\ref{subsec:resistance}.
	In Section~\ref{subsec:Markov_type} we review the theory of Markov-type inequalities and prove our upper bound on the mean-squared displacement, Theorem~\ref{thm:displacement}. Finally, in Section~\ref{subsec:Kumagai_Misumi} we show how the remaining estimates of Theorem~\ref{thm:est_collect} can be deduced from these estimates using the methods of \cite{BJKS08,kumagai2008heat}.

\subsection{Effective Resistance}
\label{subsec:resistance}
In this section we prove Theorem~\ref{cor:eff_cond}. The upper bound is trivial since resistances are always bounded by distances, so we focus on the lower bound.
We will employ \cite[Lemma 8.3]{MR4055195} which we reproduce here. Let $\mathscr{C}_{\mathrm{eff}}(A\leftrightarrow B;G)=\mathscr{R}_{\mathrm{eff}}(A\leftrightarrow B;G)^{-1}$ denote the effective conductance between sets $A,B\subseteq V[G]$.
\begin{lemma}[\!\!\cite{MR4055195}, Lemma 8.3] \label{lemma:eff_res}
	Let $T$ be a tree, let $v$ be a vertex of $T$, and let $N_v(n,k)$ be the number of vertices $u\in\partial B(v,k):=B(v,k)\setminus B(v,k-1)$ at distance $k$ from $v$  such that $u$ lies on a geodesic in $T$ from $v$ to $\partial B(v,n)$. Then 
	\[
	\mathscr{C}_{\mathrm{eff}}(v\leftrightarrow \partial B(v,n);T)\leq \frac{1}{k} N_v(n,k)
	\]
	for every $1\leq k\leq n$.
\end{lemma}

We will also use the following theorem of \cite{hutchcroft2020logarithmic} concerning the tail of the intrinsic radius of the past. For each $n\geq 0$, let $\partial\mathfrak{P}(0,n)$ be the set of vertices in $\mathfrak{P}(0)$ with an intrinsic distance from $0$ of exactly $n$.
\begin{theorem}[\!\!\cite{hutchcroft2020logarithmic}, Theorem 1.1] \label{theorem:past}
	Let $\fT$ be the uniform spanning tree of $\Z^4$. Then
	\[
	\pr(\partial\mathfrak{P}(0,n)\neq\emptyset)\asymp \frac{(\log n)^{1/3}}{n}
	\]
	for every $n\geq 1$.
\end{theorem}

We now apply these results together with Theorem~\ref{thm:main_geometric_theorem} to prove Theorem~\ref{cor:eff_cond}.

\begin{proof}[Proof of Theorem~\ref{cor:eff_cond}]
	Fix $\lambda>0$ and $\delta\in(0,1]$.
	For each $0\leq m\leq n$, let $K(n,m)$ be the set of vertices $u\in\partial \mathfrak{B}(0,m)$ that lie on a geodesic from $0$ to $\partial \mathfrak{B}(0,n)$ and let $K^\prime (n,m)$ be the set of vertices $u\in\partial \mathfrak{B}(0,m)$ such that $\partial\fP(u,n-m)\neq\emptyset$. We observe that $K(n,m)\setminus K^\prime(n,m)$ contains at most one vertex, namely the unique vertex in $\partial \fB(0,m)$ which lies in the future of $0$, and so, by Lemma \ref{lemma:eff_res}, we have 
	\[
	\mathscr{C}_{\mathrm{eff}}(0\leftrightarrow \partial \mathfrak{B}(0,n);\fT)\leq \frac{1}{m}\abs{K(n,m)}\leq\frac{1}{m}+\frac{1}{m}\abs{K^\prime(n,m)}
	\]
	for each $1\leq m\leq n$. Averaging this gives us that
	\[
	\mathscr{C}_{\mathrm{eff}}(0\leftrightarrow \partial \mathfrak{B}(0,3n);\fT)\preceq \frac{1}{n}+ \frac{1}{n^2}\sum_{m=n}^{2n}\abs{K^\prime(3n,m)},
	\]
	for each $n\geq1$.	Now, for each $n\geq 1$, the sets $(K^\prime(n,m))_{n\leq m\leq 2n}$ are pairwise disjoint and their union satisfies
	\[
	\bigcup_{n\leq m\leq 2n} K^\prime(n,m)\subseteq \{u\in\Z^4: u\in \mathfrak{B}(0,2n),\, \partial\fP(u,n)\neq\emptyset\},
	\]
	and so 
	\[
	\mathscr{C}_{\mathrm{eff}}(0\leftrightarrow \partial \mathfrak{B}(0,3n);\fT)\preceq \frac{1}{n}+ \frac{1}{n^2}\sum_{u\in\Z^4} \mathbbm{1}\big(u\in \mathfrak{B}(0,2n),\, \partial\fP(u,n)\neq\emptyset\big).
	\]
	Multiplying both sides by the indicator function $\mathbbm{1}(\abs{\mathfrak{B}(0,4n)}\leq \lambda^{1/2}n^2(\log n)^{-1/3+\delta})$ and taking expectations gives
	\begin{align*}
		\nonumber&\E{\mathscr{C}_{\mathrm{eff}}(0\leftrightarrow \partial \mathfrak{B}(0,3n);\fT)\mathbbm{1}\left(\abs{\mathfrak{B}(0,4n)}\leq \frac{\lambda^{1/2}n^2}{(\log n)^{1/3-\delta} }\right)}\\
		\nonumber	&\hspace{5cm} \preceq \frac{1}{n}+ \frac{1}{n^2}\sum_{u\in\Z^d} \pr\left(u\in \mathfrak{B}(0,2n),\, \partial\mathfrak{P}(u,n)\neq\emptyset,\, \abs{\mathfrak{B}(0,4n)}\leq \frac{\lambda^{1/2}n^2}{(\log n)^{1/3-\delta} }\right),
	\end{align*}
	and applying the mass-transport principle to exchange the roles of $0$ and $u$ yields that
	\begin{align}
		\label{eq:cond_exp_ceff}
		\nonumber&\E{\mathscr{C}_{\mathrm{eff}}(0\leftrightarrow \partial \mathfrak{B}(0,3n);\fT)\mathbbm{1}\left(\abs{\mathfrak{B}(0,4n)}\leq \frac{\lambda^{1/2}n^2}{(\log n)^{1/3-\delta} }\right)}\\
		&\hspace{5cm}\preceq \frac{1}{n}+ \frac{1}{n^2}\sum_{u\in\Z^d} \pr\left(0\in \mathfrak{B}(u,2n),\, \partial\mathfrak{P}(0,n)\neq\emptyset,\, \abs{\mathfrak{B}(u,4n)}\leq \frac{\lambda^{1/2}n^2}{(\log n)^{1/3-\delta} }\right)\nonumber\\
		\nonumber	&\hspace{5cm}\leq \frac{1}{n}+ \frac{1}{n^2}\E{\abs{\mathfrak{B}(0,2n)}\mathbbm{1}\left(\abs{\mathfrak{B}(0,2n)}\leq \frac{\lambda^{1/2}n^2}{(\log n)^{1/3-\delta} },\, \partial\mathfrak{P}(0,n)\neq\emptyset\right)}\\
		&\hspace{5cm}\preceq \frac{1}{n}+ \frac{\lambda^{1/2}}{(\log n)^{1/3-\delta}}\pr\big(\partial\mathfrak{P}(0,n)\neq\emptyset\big) \preceq \lambda^{1/2}\frac{(\log n)^\delta}{n},
	\end{align}
	where the final inequality follows from Theorem \ref{theorem:past}.
	Now by a union bound, we have
	\begin{multline*}
		\pr\left(\mathscr{C}_{\mathrm{eff}}(0\leftrightarrow \partial \mathfrak{B}(0,3n);\fT)>\lambda \frac{(\log n)^\delta}{n}\right)\leq \pr\left( \abs{\mathfrak{B}(0,4n)}> \frac{\lambda^{1/2}n^2}{(\log n)^{1/3-\delta} }  \right)\\ +\pr\left(\mathscr{C}_{\mathrm{eff}}(0\leftrightarrow \partial \mathfrak{B}(0,3n);\fT)\mathbbm{1}\left(\abs{\mathfrak{B}(0,4n)}\leq \frac{\lambda^{1/2}n^2}{(\log n)^{1/3-\delta} }\right)>\lambda \frac{(\log n)^\delta}{n}\right).
	\end{multline*}
	Applying Markov's inequality to each term on the right hand side and using \eqref{eq:cond_exp_ceff} and Theorem \ref{thm:main_geometric_theorem} to estimate the relevant expectations yields that
	\begin{equation}
		\pr\left(\mathscr{C}_{\mathrm{eff}}(0\leftrightarrow \partial \mathfrak{B}(0,3n);\fT)>\lambda \frac{(\log n)^\delta}{n}\right)\preceq_\delta \lambda^{-1}\lambda^{1/2}+\lambda^{-1/2}\preceq \lambda^{-1/2},
	\end{equation}
	and the claim follows since $\lambda,\delta>0$ were arbitrary.
\end{proof}

\subsection{Upper bounds on displacement via Markov-type inequalities}
\label{subsec:Markov_type}

In this section, we will use Markov-type inequalities \cite{MR3077911,MR1159828,MR2239346} together with the results of \cite{hutchcroft2020logarithmic}
to prove Theorem~\ref{thm:displacement}, which establishes sharp upper bounds on the expectation of the squared maximal intrinsic displacement of a random walk on the 4d UST. Markov-type inequalities were first introduced by Ball \cite{MR1159828} in the context of the Lipschitz extension problem, and have since been found to have many important applications to the study of random walk \cite{MR4146545,MR4369717,lee2020relations,MR4348678,MR3383336}. Our work is particularly influenced by that of James Lee and his coauthors \cite{MR4348678,lee2020relations,MR4369717,ding2013markov}, who pioneered the use of Markov-type inequalities to prove sharp subdiffusive estimates for random walks on fractals. We begin by quickly reviewing the general theory, including in particular the extension of the universal Markov-type inequality for planar graphs of Ding, Lee, and Peres \cite{ding2013markov} to \emph{unimodular hyperfinite} planar graphs established in \cite{MR4146545}.

\medskip

\noindent \textbf{Unimodular weighted graphs.}  A \textbf{vertex-weighted graph} is a pair $(G,\omega)$ consisting of a graph $G$ and a \textbf{weighting} on $G$, that is a function $\omega: V[G]\rightarrow[0,\infty)$. We define the \textbf{weighted graph distance} between vertices $x,y$ of a weighted graph $(G,\omega)$ by 
\[
d_\omega^G(x,y)=\inf_{x=u_0\sim \cdots\sim u_n=y, n\in\N}\sum_{i=1}^n \frac{1}{2}\big(\omega(u_i)+\omega(u_{i-1})\big).
\]
Let $\mathcal{G}_\bullet^\omega$ be the space of triples $(G,\omega,\rho)$, where $(G,\omega)$ is a locally finite vertex-weighted graph, and $\rho\in V[G]$ is a vertex known as the \textbf{root} vertex. The space $\mathcal{G}_\bullet^\omega$ is equipped with the Borel sigma algebra induced by the natural generalisation of the Benjamini-Schramm local topology \cite{Curien,AL07}
	in which two rooted, weighted graphs are considered to be close if there exist large graph-distance balls around their roots for which their respective balls admit a graph isomorphism that approximately preserves the weights. The details of this construction are not important to us and can be found in e.g.\ \cite[Section 1.2]{Curien}. Similarly, we also have the space $\mathcal{G}_{\bullet\bullet}^\omega$ of vertex-weighted graphs with an \emph{ordered pair} of distinguished vertices.
We say that a random variable $(G,\omega,\rho)$ taking values in $\mathcal{G}_\bullet^\omega$ is a \textbf{unimodular vertex-weighted graph} if it satisfies the \textbf{mass-transport principle}, i.e.\ if
\[
\E{\sum_{v\in V[G]}F(G,\omega,\rho,v)}=\E{\sum_{v\in V[G]}F(G,\omega,v,\rho)}
\]
for each Borel measurable function $F:\mathcal{G}_{\bullet\bullet}^\omega\rightarrow[0,\infty)$.
Unweighted unimodular random graphs are defined similarly; we refer the reader to \cite{AL07,Curien} for a more in-depth discussion of the local topology and unimodularity. These notions are relevant to our setting since if $K$ is the component of the origin $0$ in some translation-invariant random subgraph of $\Z^d$ then $(K,0)$ always defines a unimodular random rooted graph, so that, in particular, $(\fT,0)$ is a unimodular random rooted graph when $\fT$ is the UST of $\Z^4$. Moreover, if the weight $\omega:\Z^4\to [0,\infty)$ is computed from $\fT$ in a translation-equivariant way then the resulting weighted random rooted graph $(\fT,\omega,0)$ is also unimodular, as can be seen by applying the usual mass-transport principle on $\Z^4$ to the expectations $\mathbb{E} F(\fT,\omega,x,y)$.

\medskip

\noindent \textbf{Markov-type inequalities.}
A metric space $\mathcal{X}=(\mathcal{X},d)$ is said to have \textbf{Markov-type} $2$ with constant $c<\infty$ if 
for every finite set $S$, every irreducible reversible Markov chain $M$ on $S$, and every function $f:S\rightarrow \mathcal{X}$ the inequality
\[
\mathbb{E}\left[d\big(f(Y_0),f(Y_n)\big)^2\right]\leq c^2n \mathbb{E}\left[d\big(f(Y_0),f(Y_1)\big)^2\right]
\] 
holds for every $n\geq 0$,
where $(Y_i)_{i\geq 0}$ is a trajectory of the Markov chain $M$ with $Y_0$ distributed as the stationary measure of $M$. Similarly, a metric space $\mathcal{X}=(\mathcal{X},d)$ is said to have \textbf{maximal Markov-type} $2$ with constant $c<\infty$ if for every finite set $S$ and every irreducible reversible Markov chain $M$ on $S$, and every function $f:S\rightarrow \mathcal{X}$, we have that 
\[
\mathbb{E}\left[\max_{0\leq i\leq n}d\big(f(Y_0),f(Y_i)\big)^2\right]\leq c^2n \mathbb{E}\left[d\big(f(Y_0),f(Y_1)\big)^2\right]
\] 
for each $n\geq 0$,
where, as before, $(Y_i)_{i\geq 0}$ is a trajectory of the Markov chain $M$ with $Y_0$ distributed as the stationary measure of $M$. 

\medskip

It is proved in \cite{ding2013markov} that there exists a universal constant $C$ such that every vertex-weighted planar graph has Markov-type 2 with constant $C$; in fact their proof also establishes the existence of a universal constant $C$ such that every weighted planar graph has \emph{maximal} Markov-type 2 with constant $C$ as explained in \cite[Proposition 2.4]{MR4146545}.  This fact is significantly easier for \emph{trees}, where it was established by Naor, Peres, Schramm, and Sheffield \cite{MR2239346} (see also \cite[Theorem 13.14]{LP:book}).

\medskip

We now describe the consequences of this theorem for unimodular random planar graphs. We must first define what it means for a unimodular random rooted graph to be \emph{hyperfinite}.
A \textbf{percolation} on a unimodular random rooted graph $(G,\rho)$ is a labelling $\eta$ of the edge set of $G$ by the elements ${0,1}$ such that the resultant edge-labelled graph $(G,\eta,\rho)$ is unimodular. We think of the percolation $\eta$ as a random subgraph of $G$, where each edge is labelled $1$ if it is included in the subgraph and $0$ otherwise, and denote the connected component of $\rho$ in this subgraph as $K_\eta(\rho)$. We say a percolation is finitary if $K_\eta(\rho)$ is almost surely finite, and say a unimodular random rooted graph $(G,\rho)$ is  \textbf{hyperfinite} if there exists an increasing sequence of finitary percolations $(\eta_n)_{n\geq 1}$ such that $\cup_{n\geq 1} K_{\eta_n}(\rho)=V[G]$ almost surely. The component of the origin in a translation-invariant random subgraph of $\Z^d$ is always hyperfinite as can be seen by taking a random hierarchical partition of $\Z^d$ into dyadic boxes. The following proposition appears as \cite[Corollary 2.5]{MR4146545}.

\begin{proposition} \label{prop:hyper_markov}
	Let $(G,\rho)$ be a hyperfinite, unimodular random rooted graph with $\E{\deg(\rho)}<\infty$ that is almost surely planar, and suppose that $\omega$ is a vertex-weighting of $G$ such that $(G,\omega,\rho)$ is a unimodular vertex-weighted graph. If $Y$ is a random walk on $G$ started at $\rho$ then
	\[
	\mathbb{E}\left[\deg(\rho)\max_{0\leq i\leq n}d_\omega^G\big(Y_0,Y_i\big)^2\right]\leq C^2 n \mathbb{E}\left[\deg(\rho)\omega(\rho)^2\right],
	\] 
	for each $n\geq 1$, where $C$ is a universal constant.
\end{proposition}

We now apply this proposition to prove Theorem~\ref{thm:displacement}.

\begin{proof}[Proof of Theorem~\ref{thm:displacement}]
	Let $r\geq 1$ be a parameter to be optimized over shortly.
	Seeing as the UST of $\Z^d$ is unimodular, hyperfinite (being a translation-invariant percolation processes on $\Z^d$) and planar, we can apply Proposition \ref{prop:hyper_markov} to the vertex weight
	\[
	\omega_r(v)=\mathbbm{1}(\partial\mathfrak{P}(v,r)\neq\emptyset),
	\]
	which makes $(\fT,\omega_r,0)$ unimodular since it is computed as a translation-equivariant function of $\fT$.
	This particular choice of weight is inspired by that used by Ganguly and Lee in \cite{MR4369717}.
	Writing $d_r=d_{\omega_r}^\fT$ and using the fact that $\fT$ has degrees uniformly bounded below by $1$ and above by $8$, we get that 
	\begin{equation}\label{eq:MarkovT_applied}
		\mathbb{E}\left[\max_{0\leq i\leq n}d_{r}\big(Y_0,Y_i\big)^2\right]\leq 8C^2 n\pr\big(\partial\mathfrak{P}(0,r)\neq\emptyset\big)
	\end{equation}
	for each $r,n \geq 1$.
	We next claim that
		\begin{equation}
			\label{eq:weighted_distance_claim}
			d_\fT(u,v)\leq 4r + 4d_r(u,v) \qquad \text{for every $u,v\in \fT$ and $r\geq 1$.}
		\end{equation}
		Indeed, let $u,v\in \Z^4$ and suppose that $d_\fT(u,v) \geq 4r$, the claimed inequality being trivial otherwise. Let $w$ be the vertex at which the futures of $u$ and $v$ meet. At least one of the inequalities $d_{\mathfrak{T}}(u,w) \geq \frac{1}{2}d_\fT(u,v)$ or $d_{\mathfrak{T}}(v,w) \geq \frac{1}{2}d_\fT(u,v)$
		holds, and we may assume without loss of generality that $d_{\mathfrak{T}}(u,w) \geq \frac{1}{2}d_\fT(u,v) \geq 2r$. Since $u$ belongs to the past of each of the vertices in the $\fT$-geodesic connecting $u$ to $w$, all the vertices in the second half of this geodesic must have past of intrinsic diameter at least $r$, so that $d_r(u,w) \geq \frac{1}{2} d_\fT(u,w)$ and hence that $d_r(u,v) \geq \frac{1}{4}d_\fT(u,v)$ as required.
	It follows from \eqref{eq:weighted_distance_claim} together with \eqref{eq:MarkovT_applied} that
	\[
	\mathbb{E}\left[\max_{0\leq i\leq n}d_{\mathfrak{T}}\big(Y_0,Y_i\big)^2\right]
	\leq 32r^2 + 32\mathbb{E}\left[\max_{0\leq i\leq n}d_{r}\big(Y_0,Y_i\big)^2\right]
	\preceq r^2+ n\pr\big(\partial\mathfrak{P}(0,r)\neq\emptyset\big) \preceq r^2+ \frac{n(\log r)^{1/3}}{r}
	\]
	for every $r,n\geq 1$, where we applied Theorem \ref{theorem:past} in the third inequality, and taking $r=\lceil n^{1/3}(\log n)^{1/9}\rceil$ yields that
	\[
	\mathbb{E}\left[\max_{0\leq i\leq n}d_{\mathfrak{T}}\big(Y_0,Y_i\big)^2\right]\preceq n^{2/3}(\log n)^{2/9}
	\]
	for every $n\geq 2$ as claimed.
\end{proof}
\begin{remark}
	This method also gives sharp upper bounds in dimensions $d\geq 5$: applying \cite[Theorem 1.2]{MR4055195} in place of Theorem~\ref{theorem:past}, it yields that if $d\geq 5$, $\fT$ is the component of the origin in the uniform spanning forest of $\Z^d$, and $Y$ is a random walk on $\fT$ started at $0$, then
	\[
	\mathbb{E}\left[\max_{0\leq i\leq n}d_{\mathfrak{T}}\big(Y_0,Y_i\big)^2\right]\preceq n^{2/3}
	\] 
	for every $n\geq 0$. This is stronger than the displacement upper bounds proven in \cite{MR4055195}, which were based on the results of \cite{BJKS08}.
\end{remark}

\subsection{Proof of Theorem \ref{thm:est_collect}}
\label{subsec:Kumagai_Misumi}

In this section we use all of the previous results to compute logarithmic corrections to the asymptotic behaviour of the displacement, exit times, return probabilities and range of the simple random walk on the uniform spanning tree. We will draw heavily on the methods of \cite{kumagai2008heat}, which generalizes and synthesizes the earlier works \cite{MR2177164,BJKS08,MR2247823}. Note that we must rederive all our results from the methods of \cite{kumagai2008heat} rather than simply quote their results since, as stated, these results do not allow for non-matching upper and lower bounds.


	\begin{remark}
		In this proof we will often use our big-O in probability notation on random variables indexed by more than one variable (e.g. $n$ and $r$). When we write an expression $X_{n,r}=\bO(Y_{n,r})$ of this form, it means that the entire family of associated random variables indexed by both $n$ and $r$ is tight.
	\end{remark}

\begin{proof}[Proof of Theorem~\ref{thm:est_collect}]
	We recall that $\mathbf{E}^{\mathfrak{T}}_x$ denotes expectation with respect to the law of a simple random walk $X$ on $\mathfrak{T}$ started at $x\in \Z^4$ conditional on $\mathfrak{T}$, and write $\mathbf{P}^{\mathfrak{T}}_x$ for the corresponding probability measure. 
	Where clear from context, we will write $\mathbb{P}$ for the joint law and expectation of the uniform spanning tree and a random walk on the tree started at the origin, and similarly will write $\mathbb{E}$ for expectation with respect to this joint law.
	\medskip

	\noindent
	\textbf{Heat-kernel upper bound:} 
	\cite[Proof of Proposition 3.1(a)]{kumagai2008heat}
	implies that
	\[
	p_{2n}^{\fT}(0,0)+p_{2n+1}^{\fT}(0,0)\preceq \frac{1}{\abs{\mathfrak{B}(0,R)}}\vee\frac{R}{n}
	\]
	for every $n,R\geq 1$.
	Taking $R=n^{1/3}(\log n)^{1/9}$ and applying the volume lower bound of Theorem \ref{thm:main_geometric_theorem} therefore yields that
	\begin{equation}
		\label{eq:HK_upper}
		p_{2n}^{\fT}(0,0)=\textbf{O}\left(\frac{(\log n)^{1/9}}{n^{2/3}}\right)
	\end{equation}
	for every $n \geq 2$.
	
	\medskip
	
	\noindent
	\textbf{Intrinsic displacement lower bound:} 
	We have by Cauchy-Schwarz that
	\begin{multline} \label{eq:p_CS}
		\pb^{\fT}_0(d_{\fT}(0,X_n)\leq r) = \sum_{v\in \fB(0,r)}{p^\fT_n}(o,v)\leq{\abs{\fB(0,r)}}^{1/2} \left(\sum_{v\in \mathfrak{B}(R)}{p^\fT_n}(o,v)^2\right)^{1/2}
		\\\preceq {\abs{\fB(0,r)}}^{1/2} p_{2n}^{\fT}(0,0)^{1/2}=\bO\left(\frac{r}{(\log r)^{1/6-o(1)}} \cdot \frac{1}{n^{\frac{1}{3}}} (\log n)^{\frac{1}{18}}\right)
	\end{multline}
	for every $n,r\geq 1$, where we have applied the volume upper bound on Theorem \ref{thm:main_geometric_theorem}, and the previously derived heat-kernel upper bounds. If we take $r = n^{1/3} (\log n)^{1/9-\delta}$ for some $\delta>0$, then the expression appearing inside the $\bO$ is $o(1)$, and, since this holds for every $\delta>0$ (with implicit constants depending on $\delta$), it follows that
	$d_{\fT}(0,X_n) =\bOmega( n^{1/3} (\log n)^{1/9-o(1)})$ for every $n\geq 2$ as claimed.
	
	\medskip

		\noindent \textbf{Intrinsic displacement upper bound:} The estimate
		\[
		d_\fT(X_0,X_n) \leq \max_{0\leq m \leq n} d_\fT(X_0,X_m) = \bO\left(n^{1/3}(\log n)^{1/9}\right)
		\]
		follows immediately from Theorem~\ref{thm:displacement}.
	
	\medskip

	\noindent\textbf{Heat-kernel lower bound: }
	Fix $\delta>0$ and let $R=n^{1/3}(\log n)^{1/9+\delta}$.
	Using the same Cauchy-Schwarz argument as in \eqref{eq:p_CS}, it follows from the intrisic displacement upper bounds of Theorem~\ref{thm:displacement} and the volume lower bounds of Theorem \ref{thm:main_geometric_theorem} that there exists $N_\delta$ such that
	\[
	p_{2n}^{\fT}(o,o) \geq \frac{(1-\pb^{\fT}(d_{\fT}(o,X_n)>R))^2}{\abs{\mathfrak{B}(0,R)}}=\frac{1-\bo(1)}{\mathbf{O}\big(R^2(\log R)^{-1/3}\big)}
		=
		\frac{1-\bo(1)}{\mathbf{O}\big(n^{2/3}(\log n)^{-1/9+2\delta}\big)}
	\]
	for every $n\geq N_\delta$, and the claim follows since $\delta>0$ was arbitrary.
	
	\medskip
	
	\noindent \textbf{Exit time upper bound:}
	\cite[Equation 3.7]{kumagai2008heat} implies that
	\[
	\mathbf{E}^\fT_0[\tau_R]\leq \mathscr{R}_{\mathrm{eff}}(0\leftrightarrow {\mathfrak{B}(0,R)^c;\fT}) \abs{\mathfrak{B}(0,R)}\leq R\abs{\mathfrak{B}(0,R)}
	\]
	for every $R\geq 1$, and applying
	Theorem \ref{thm:main_geometric_theorem} yields that 
	\[
	\mathbf{E}^\fT[\tau_R]=\textbf{O}\left(\frac{R^3}{(\log R)^{1/3-o(1)}}\right) \qquad \text{ and hence that } \qquad \tau_R=\textbf{O}\left(\frac{R^3}{(\log R)^{1/3-o(1)}}\right)
	\]  
	for every $R\geq 2$.
	
	\medskip
	
	\noindent \textbf{Exit time lower bound: }
	Fix $R\geq 1$, and let $\beta>0$, $n= R^3/(\log R)^{1/3}$. Applying Theorem \ref{thm:displacement}, we have 
	\[
	\pr(\tau_R\leq \beta n)=\pr\left(\max_{0\leq i\leq \beta n} d_{\mathfrak{T}}(o,X_i)^2\geq R^2\right)=O\left( \frac{\beta^{2/3}n^{2/3}(\log n)^{2/9}}{R^2}\right)=O(\beta^{2/3}),
	\] and so $\tau_R=\mathbf{\Omega}(R^3/(\log R)^{1/3})$. The relation $\bE^{\fT}[\tau_R]=\mathbf{\Omega}(R^3/(\log R)^{1/3})$ then follows. 
	
	\medskip

	\noindent \textbf{Extrinsic displacement upper bound:} Let $R\geq 1$ and fix $\delta>0$. We have already established that 
	\[\max_{0\leq m \leq n} d_\fT(X_0,X_m) = \bO\left(n^{1/3}(\log n)^{1/9}\right),\]
	and Theorem \ref{thm:main_geometric_theorem} tells us that
	\[
	\mathfrak{B}(n)\subseteq \Lambda\big(n^{1/2}(\log n)^{1/6+\mathbf{o}(1)}\big) \qquad \text{as $n\to\infty$}.
	\]
	Combining these two facts gives us 
	\[\max_{0\leq m \leq n} \norm{X_m}_\infty =\mathbf{O}\bigl(n^{\frac{1}{6}}(\log n)^{\frac29+ o(1)} \bigr) \qquad \text{as $n\to\infty$,}
	\]
	as required.

		\medskip
		
	\noindent \textbf{Extrinsic displacement lower bound:}
 Let $R\geq 1$. Exploiting the tree structure of $\fT$, we note that if $\max_{m\leq n}\norm{X_m}_\infty\leq R$, then $\Gamma(0,X_n)\subseteq \Lambda(R)$. Thus, arguing as in \eqref{eq:p_CS}, we have that
	\begin{align*}
		\pb^{\fT}\big(\max_{m\leq n} \norm{X_m}_\infty\leq R\big) &\preceq \abs{\{x\in\Z^d:\Gamma(0,x)\subseteq \Lambda( R)\}}^{1/2} p_{2n}^{\fT}(o,o)^{1/2}\\
		&=\textbf{O}\left(\frac{R^2}{(\log R)^{1/2}}\cdot \frac{(\log n)^{1/18}}{n^{1/3}}\right),
	\end{align*}
	where the we have applied Proposition \ref{prop:extr_volume} and heat kernel upper bound \eqref{eq:HK_upper} in the last line. This implies that $\max_{m\leq n} \norm{X_m}_\infty  = \bOmega(n^{1/6}(\log n)^{2/9})$ as claimed.
	
	\medskip
	
	\medskip
	\noindent \textbf{Range upper bound: }
	Fix $\delta>0$. For $n\geq 1$, let $D_n = \max_{0\leq i\leq n} d_{\mathfrak{T}}(0,X_i)$. Applying displacement upper bounds and the volume upper bounds of Theorem \ref{thm:main_geometric_theorem}, we have that
	\[
	\abs{\{X_m:0\leq m \leq n\}}\leq \left
	\vert\mathfrak{B}(D_n)\right\vert=\left\vert\mathfrak{B}(\mathbf{O}(n^{1/3}(\log n)^{1/9}))\right\vert=\mathbf{O}\left(\frac{n^{2/3}}{(\log n)^{1/9-o(1)}}\right)
	\]
	as $n\to\infty$ as required.

	\medskip
	
	\noindent \textbf{Range lower bound: }
	Fix $R\geq 1$, $\delta>0$ and write $\mathfrak{B}=\mathfrak{B}(R)$. Let $g_R(x,y)=(\deg_\fT y)^{-1} \mathbf{E}_x^{\mathfrak{T}}[{\sum_{0\leq i \leq \tau_R} \mathbbm{1}(X_n=y)}]$ and let $p(y)=g_R(0,y)/g_R(y,y)$ be the probability that a random walk started at $0\in\mathfrak{T}$ hits $y$ before exiting $\mathfrak{B}$. For each $y\in\mathfrak{B}^\prime:=\mathfrak{B}(\lfloor R/(\log R)^\delta \rfloor)$, we have $\mathscr{R}_{\mathrm{eff}}(0\leftrightarrow y;\mathfrak{T})\leq R/(\log R)^\delta$, so that if the event  $A=\{\mathscr{R}_{\mathrm{eff}}(0\leftrightarrow \mathfrak{B}^c;\mathfrak{T})\geq R/(\log R)^{\delta/2}\}$ holds then
	\begin{multline*}
	\inf_{y\in\mathfrak{B}^\prime}\mathscr{R}_{\mathrm{eff}}(y\leftrightarrow \mathfrak{B}^c;\mathfrak{T})\geq \inf_{y\in\mathfrak{B}^\prime} \big[\mathscr{R}_{\mathrm{eff}}(0\leftrightarrow \mathfrak{B}^c;\mathfrak{T})-\mathscr{R}_{\mathrm{eff}}(0\leftrightarrow y;\mathfrak{T})\big]\\ 
	\geq R/(\log R)^{\delta/2}-R/(\log R)^\delta=\Omega(R/(\log R)^{\delta/2}).
	\end{multline*}
	Now for each $y\in\mathfrak{B}$ we have the following inequality which was derived for general graphs in \cite[Proof of Proposition 3.2(b)]{kumagai2008heat}:
	\[
	\abs{1-p(y)}^2\leq \mathscr{R}_{\mathrm{eff}}(0\leftrightarrow y;\mathfrak{T})\mathscr{R}_{\mathrm{eff}}(y\leftrightarrow \mathfrak{B}^c;\mathfrak{T})^{-1}.
	\]
	Taking the supremum over $y\in\mathfrak{B}^\prime\subset\mathfrak{B}$ yields
	\begin{equation*}
		\sup_{y\in\mathfrak{B}^\prime}\abs{1-p(y)}^2\leq \frac{R}{(\log R)^\delta}\cdot	\sup_{y\in\mathfrak{B}^\prime}\mathscr{R}_{\mathrm{eff}}(y\leftrightarrow \mathfrak{B}^c;\mathfrak{T})^{-1}=O((\log R)^{-\delta/2})
	\end{equation*}
	on the event $A$.
  For each $R\geq 1$, consider the random variable $U_R=\abs{\{X_i:0\leq i \leq \tau_R\}\cap \mathfrak{B}^\prime}$. Then
	\begin{equation}\label{eq:lower_TR}
		\mathbf{E}_0^{\mathfrak{T}}{[U_R]}\geq\mathbf{E}^{\mathfrak{T}}_0\Big[\sum_{x\in\mathfrak{B}^\prime}\mathbbm{1}(X\text{ hits $x$ before exiting $\mathfrak{B}$})\Big]=\sum_{y\in\mathfrak{B}^\prime}p(y)\geq \mathbbm{1}(A)(1-O((\log R)^{-\delta/4}))\abs{\mathfrak{B}^\prime}.
	\end{equation}
	Now
	\[
	\pr\left(\frac{U_R}{\abs{\mathfrak{B}^\prime}}\leq 1/2\right)\leq\mathbb{E}\left[\pb^{\mathfrak{T}}\left(A,\frac{U_R}{\abs{\mathfrak{B}^\prime}}\leq 1/2\right)\right]+\pr(A^c)=\mathbb{E}\left[\pb^{\mathfrak{T}}\left(\mathbbm{1}(A)\Big(1-\frac{U_R}{\abs{\mathfrak{B}^\prime}}\Big)\geq 1/2\right)\right]+\pr(A^c),
	\]
	and so applying \eqref{eq:lower_TR} with Markov's inequality to the conditional probability inside the expectation gives
	\[
	\pr\left(\frac{U_R}{\abs{\mathfrak{B}^\prime}}\leq 1/2\right)\leq O((\log R)^{-\delta/4})\pr(A)+ \pr(A^c)=o(1)
	\]
	as $R\rightarrow\infty$, where the fact that $\pr(A^c)\to 0$ as $R\to \infty$ follows from Corollary \ref{cor:eff_cond}. 
  The claim follows since $|\fB'|=\bOmega(R^2(\log R)^{-1/3-2\delta})$, $\tau_R=\bO(R^3(\log R)^{-1/3+o(1)})$, and $\delta>0$ was arbitrary.
\end{proof}

\subsection*{Acknowledgements}
We thank Sebastian Andres for helpful discussions. TH also thanks Ben Golub for references on the `big-O and little-o in probability' notation.

\addcontentsline{toc}{section}{References}

 \setstretch{1}
 \footnotesize{
  \bibliographystyle{abbrv}
  \bibliography{unimodularthesis.bib}
  }

\medskip

\newcommand{\email}[1]{\href{mailto:#1}{#1}}

\noindent \sc{N.\ Halberstam: CCIMI, University of Cambridge,} \email{nh448@cam.ac.uk}\\
\noindent \sc{T.\ Hutchcroft: PMA, Caltech} \email{t.hutchcroft@caltech.edu} 

\end{document}